\documentclass[a4paper,10pt]{article}
\usepackage[utf8]{inputenc}
\usepackage[a4paper, total={7in, 9in}]{geometry}
\usepackage{amssymb}
\usepackage{amsmath}
\usepackage{graphicx,amsthm,color}
\usepackage{diagbox}
\usepackage{subfig}
\usepackage{float}
\usepackage{hyperref,cleveref,cite,url}
\hypersetup{
  colorlinks,
  citecolor=green!50!black,
  linkcolor=red,
  urlcolor=blue!50!black}
\usepackage{makecell}
\usepackage{pgf,tikz}
\usetikzlibrary{arrows}
\usepackage{authblk}

\newtheorem{theorem}{Theorem}
\newtheorem{proposition}[theorem]{Proposition}
\newtheorem{observation}[theorem]{Observation}
\newtheorem{corollary}[theorem]{Corollary}
\newtheorem{lemma}[theorem]{Lemma}

\theoremstyle{definition}

\newtheorem{definition}[theorem]{Definition}
\newtheorem{conjecture}[theorem]{Conjecture}

\newcommand{\tw}{\mathrm{tw}}
\newcommand{\dg}{\mathrm{deg}}

\makeatletter
\providecommand\phantomcaption{\caption@refstepcounter\@captype}
\makeatother

\title{Feedback vertex sets in (directed) graphs of bounded degeneracy or treewidth}
\author[1]{Kolja Knauer}
\author[2]{Hoang La}
\author[2]{Petru Valicov}
\affil[1]{Departament de Matem\`atiques i Inform\`atica,
Universitat de Barcelona (UB), Barcelona, Spain\\Aix Marseille Univ, Universit\'e de Toulon, CNRS, LIS, Marseille, France}
\affil[2]{LIRMM, Université de Montpellier, CNRS, Montpellier, France}

\begin{document}

\maketitle

\begin{abstract}
We study the minimum size $f$ of a feedback vertex set in directed and undirected $n$-vertex graphs of given degeneracy or treewidth.  In the undirected setting the bound $\frac{k-1}{k+1}n$ is known to be tight for graphs with bounded treewidth $k$ or bounded odd degeneracy $k$. We show that neither of the easy upper and lower bounds $\frac{k-1}{k+1}n$ and $\frac{k}{k+2}n$ can be exact for the case of even degeneracy. More precisely, for even degeneracy $k$ we prove that $f < \frac{k}{k+2}n$ and for every $\epsilon>0$, there exists a $k$-degenerate graph for which $f\geq \frac{3k-2}{3k+4}n -\epsilon$.

For directed graphs of bounded degeneracy $k$, we prove that $f\leq\frac{k-1}{k+1}n$ and that this inequality is strict when $k$ is odd. For directed graphs of bounded treewidth $k\geq 2$, we show that $f \leq \frac{k}{k+3}n$ and for every $\epsilon>0$, there exists a $k$-degenerate graph for which $f\geq \frac{k-2\lfloor\log_2(k)\rfloor}{k+1}n -\epsilon$. Further, we provide several constructions of low degeneracy or treewidth and large $f$.
\end{abstract}

\section{Introduction}

We consider only simple graphs and oriented directed graphs, i.e, our graphs do not have loops or multiple edges or arcs, not even anti-parallel arcs. A set $F\subseteq V$ of vertices of a (directed) graph, is a \emph{feedback vertex set} if deleting $F$ results in a (directed) graph without (directed) cycles. The complement of a feedback vertex set is called \emph{acyclic set}, and some results in the literature are formulated in terms of acyclic sets. Deciding whether a graph has a feedback vertex set of a given size is among the 21 original NP-complete problems of Karp~\cite{K72}. Thus finding the minimum size of a feedback vertex set or equivalently, the largest acyclic set, is a challenging algorithmic problem and was extensively studied in the literature.

Because of its hardness, a natural class to study the minimum size of a feedback vertex set are sparse (directed) graphs. A particular example are planar graphs.
The size of a minimum feedback vertex set in a planar graph is (famously) conjectured to be at most half the vertices by Albertson and Berman~\cite{AB76}. Up to date the best-known upper bound is $\frac{3}{5}n$ achieved through acyclic colorings with Borodin's result~\cite{B76}.
The size of a minimum feedback vertex set in planar directed graphs has also been studied by several authors. Using results of Esperet, Lemoine and Maffray~\cite{ELM17} and Li and Mohar~\cite{LM17} one can infer different upper bounds depending on the length of a shortest directed cycle. However, a question of Albertson~\cite{W06,M02} asking whether any planar directed graphs has a feedback vertex set on at most half its vertices, remains open. Note that this latter question is a weakening of the undirected setting as well as of famous Neumann-Lara conjecture~\cite{NL85}. Further, it is known that if true this bound is best-possible~\cite{KVW17}. Moreover, it is noteworthy that the best known upper bound coincides with the above mentioned $\frac{3}{5}n$ from the undirected setting~\cite{B76}.

Another class that has received attention in the directed setting  are tournaments. Already Stearns~\cite{S59} and Erd\H{o}s and Moser~\cite{EM64} have shown that any tournament on $n$ vertices admits a feedback vertex set of size $n-\lfloor\log_2(n)\rfloor-1$, while there are tournaments where no feedback vertex set on less than $n-2\lfloor\log_2(n)\rfloor-1$ vertices exists. More precise bounds for small values of $n$ have been obtained by Sanchez-Flores~\cite{SF94,SF98} and recently more work has been done into that direction by Neiman, Mackey and Heule~\cite{NMH20} and by Lidick\'y and Pfender~\cite{LP2021}.
Improving the asymptotic upper and lower bounds remains an open problem.

In this paper we focus on the class of (directed) graphs of bounded treewidth or degeneracy. Here, the treewidth or degeneracy of a directed graph is simply the treewidth or degeneracy of its underlying undirected graph. Recall that every graph of treewidth $k$ also has degeneracy $k$. 
In the undirected setting, the minimum feedback vertex set of graphs of bounded treewidth has been determined by Fertin, Godard and Raspaud~\cite{FGR02}: for a graph of order $n$, treewidth $k$, the size of a minimum feedback vertex set is at most $\frac{k-1}{k+1}n$ and this bound is best-possible. Moreover, for odd degeneracy $k$ it is easy to achieve the same upper bound, see \Cref{prop:un_dg_trivial}. However, for even degeneracy the same argument only yields an upper bound of $\frac{k}{k+2}n$, and a lower bound of $\frac{k-1}{k+1}n$. Indeed, in~\cite{BDBS14} Borowiecki, Drgas-Burchardt, and Sidorowicz show that the true value for $k=2$ is $\frac{2}{5}n$ which lies strictly between the above bounds. Our main contribution here is to construct for any even $k$ a family of graphs of degeneracy $k$, whose members of large order $n$ have minimum feedback vertex sets whose size comes arbitrarily close to $\frac{3k-2}{3k+4}n$ (see \Cref{thm:lbdg}). On the other hand we know that there exists no graph of order $n$ and even degeneracy $k$ whose minimum feedback vertex set is of size $\frac{k}{k+2}n$, see \Cref{prop:notequal}.

In the directed setting to our knowledge, apart from the above mentioned results in planar digraph and tournaments no classes of given degeneracy or treewidth have been studied previously. We give an upper bound for the smallest feedback vertex sets of $n$-vertex graphs of degeneracy $k$, namely $\frac{k-1}{k+1}n$  (\Cref{thm:degdir}), which is for $k=2$ and $k=3$ yields tight bounds $\frac{1}{3}n$ and $\frac{1}{2}n$, respectively. For $k=2$, the directed triangle is a simple example reaching the upper bound and for $k=3$, the construction from~\cite{KVW17} yields $\frac{1}{2}n$ for degeneracy $3$. Unlike the undirected setting, we know that there exists no graph of order $n$ and odd degeneracy $k$ whose minimum feedback vertex set is of size $\frac{k-1}{k+1}n$ (see \Cref{prop:degdir}). We present constructions for digraphs with large minimum feedback vertex set and given small degeneracy (resp. treewidth) in \Cref{tab:directed_degen} (resp. \Cref{tab:directed_tw}), that improve on the bounds obtained from using just tournaments from~\cite{SF94,SF98,NMH20}.
For general treewidth, taking disjoint unions of the tournaments of~\cite{EM64} on can find $n$-vertex digraphs of treewidth $k$ and $f\geq \frac{k-2\lfloor\log_2(k+1)\rfloor}{k+1}n$. However, in \Cref{prop:lbtw} we show that on general directed graphs of treewidth $k$ one can force slightly larger minimum feedback vertex sets. On the other hand, in \Cref{thm:ubtw} we show that every $n$-vertex digraph of treewidth $k$ has a feedback vertex set of size at most $\frac{k}{k+3}n$.

Many of our constructions are based on two general ideas, that can be applied in the directed and undirected setting alike and may be of independent interest. See \Cref{prop:gen2,prop:gen}.

\section{Preliminaries}

\paragraph{\textbf{Definitions and notations}:} Let $G=(V,E)$ be a (directed) graph. For a vertex $u\in V$, let $N_G(u)$ be the set of vertices (neighbors) adjacent to $u$. Let $N^-_G(u)$ (resp. $N^+_G(u)$) be the set of in-neighbors (resp. out-neighbors) of $u$ when $G$ is directed. We denote by $N_G[u]:=N(u)\cup\{u\}$ the closed neighborhood of $u$. The closed neighborhood of a set of vertices $S\subseteq V$ is $N_G[S]=\bigcup_{\{v\in S\}}N_G[v]$. We define the neighborhood $N_G(S)$ of $S$ as $N_G(S)=N_G[S]\setminus S$. We define the degree $d_G(u):=|N_G(u)|$, the in-degree $d^-_G(u):=|N^-_G(u)|$, and out-degree $d^+_G(u):=|N^+_G(u)|$. We denote $\delta(G):=\min\{d_G(v)|v\in V\}$ the minimum degree of $G$. We also define the minimum in-degree $\delta^-(G)$ and out-degree $\delta^+(G)$ of $G$ similarly. We will drop the subscript $_G$ when the graph is clear from the context. A $k$-vertex (resp. $k^-$-vertex, $k^+$-vertex) is a vertex of degree $k$ (resp. at most $k$, at least $k$). A set $S$ of vertices form a \emph{clique} when every two distinct vertices of $S$ are adjacent. For every set $S\subseteq V$, we denote by $G-S$ the graph $G$ where we removed the vertices of $S$ along with their incident edges. We denote the set of integers $\{i,i+1,\dots,j\}$ by $[i;j]$ and we simplify this notation to $[j]$ when $i=1$.

\begin{definition}[$k$-elimination ordering]
Let $G=(V,E)$ be a graph, and $\phi:V \hookrightarrow [|V|]$ be an ordering of $V$. We say that $u$ precedes $v$ in $\phi$ if and only if $\phi(u)<\phi(v)$. For every vertex $v$, we define $d_p(v)=|\{u\in V\mid\phi(u)<\phi(v), uv \in E(G)\}|$ and $d_s(v)=|\{w\in V|\phi(v)<\phi(w), vw \in E(G)\}|$. 
We say that $\phi$ is a $k$-elimination ordering of $V$ if $d_p(v)\leq k$ for all $v\in V$. We call a $k$-elimination ordering \emph{chordal} if $\{u\in V\mid\phi(u)<\phi(v), uv \in E(G)\}$ is a clique for every $v\in V$.
\end{definition}

We visualize an elimination ordering often as ordering the vertices from left to right, and eliminating them from right to left. 

\begin{definition}
Let $G$ be a (undirected) graph:
\begin{itemize}
\item $G$ is $k$-degenerate if and only if $G$ has a $k$-elimination ordering.
\item $G$ is a maximal $k$-degenerate graph if and only if $G$ has a $k$-elimination ordering where $d_p(v)=\min(k,\phi(v)-1)$ for every $v\in V(G)$.
\item $G$ is \emph{chordal} if  and only if $G$ has a chordal $k$-elimination ordering.
\item $G$ is a $k$-tree if and only if $G$ has a chordal $k$-elimination ordering where $d_p(v)=\min(k,\phi(v)-1)$ for every $v\in V(G)$.
\item $G$ has treewidth $k$ if and only if $G$ is a subgraph of a $k$-tree. 
\end{itemize}
\end{definition}

\begin{observation}
Let $G$ be a (undirected) graph: 
\begin{itemize}
\item For every $v\in V(G)$, $d(v)=d_p(v)+d_s(v)$ for every ordering $\phi$.
\item If $G$ is a $k$-degenerate graph, then every subgraph of $G$ is a $k$-degenerate graph.
\item If $G$ is a $k$-tree, then $G$ is a maximal $k$-degenerate graph.
\item If $G$ is a $k$-tree and $u$ is a $k$-vertex, then $G-\{u\}$ is a $k$-tree.
\end{itemize}
\end{observation}

Given a graph $G$ we denote by $\dg(G)$ the degeneracy of $G$ and by $\tw(G)$ its treewidth. If $D$ is a directed graph, then $\dg(D)$ and $\tw(D)$ are the degeneracy and the treewidth of the underlying undirected graph $\underline{D}$, respectively. 

We denote by $n(G)$ (or simply $n$ when there is no ambiguity) the number of vertices of a graph $G$. We denote by $f(G)$ (or simply $f$ when there is no ambiguity) the minimum size of a feedback vertex set in a (directed) graph $G$. Clearly, $f(D)\leq f(\underline{D})$ for any directed graph $D$. 
We recall below the best known results on acyclic sets in tournaments.

\begin{theorem}[\!\!\cite{SF98,NMH20}] \label{tournaments}
Denote by $a(n)$ the minimum size of a maximum acyclic set among all tournaments (or equivalently all digraphs) on $n$ vertices. Then:
\begin{itemize}
\item $a(n)=3$ for $4\leq n\leq 7$,
\item $a(n)=4$ for $8\leq n\leq 13$,
\item $a(n)=5$ for $14\leq n\leq 27$,
\item $a(n)=6$ for $28\leq n\leq 34$,
\item $6\leq a(n)\leq 7$ for $34\leq n\leq 46$,
\item $a(n)=7$ for $n = 47$.
\end{itemize}
\end{theorem}

\begin{theorem}[\!\!\cite{EM64}] \label{EM}
There exists a tournament on $n$ vertices where every acyclic subset has size at most $2\lfloor\log(n)\rfloor+1$.
\end{theorem}

These results for acyclic sets translate directly to lower bounds on minimum feedback vertex sets of graphs with bounded treewidth as tournaments of size $k+1$ are $k$-trees.

\begin{corollary}[\!\!\cite{SF98,NMH20}]\label{cor:tournaments}
For all tournaments on $k+1$ vertices (of treewidth $k$), there exists a feedback vertex set of size at most $f_k$ where:
\begin{itemize}
\item $f_k = k-2$ for $3\leq k\leq 6$,
\item $f_k = k-3$ for $7\leq k\leq 12$,
\item $f_k = k-4$ for $13\leq k\leq 26$,
\item $f_k = k-5$ for $27\leq k\leq 33$,
\item $f_k \leq k-6$ for $k\geq 34$.
\end{itemize}
For every $3\leq k\leq 33$, there exists a tournament on $k+1$ vertices for which the minimum feedback vertex set has size exactly $f_k$.
\end{corollary}

\begin{corollary}[\!\!\cite{EM64}]\label{cor:EM}
There exists a tournament on $k+1$ vertices where every feedback vertex set has size at least $k-2\lfloor\log(k+1)\rfloor$.
\end{corollary}

\section{General constructions}
Here we present two constructions that work for both directed and undirected graphs. They yield families of (directed) graphs with a controllable degeneracy or treewidth and sometimes an interesting ratio $\frac{f}{n}$.
In both constructions the following definition is important:
Given a (directed) graph $G=(V,E)$, a set of vertices $R\subseteq V$ is called \emph{bad} if it is not contained in any minimum feedback vertex set. Also, for a graph and a subset $S\subset V$ a (chordal) $k$-elimination ordering $\phi$ is called \emph{$S$-last} if it satisfies $\phi(u)<\phi(v)$ for all $u\in S$ and $v\in V\setminus S$.

%


\begin{proposition}\label{prop:gen2}
 Let $D_0$ be a (directed) graph on $k$ vertices, $D_1,\ldots,D_k$ be (directed) graphs with minimum feedback vertex sets of size $f_0, f_1,\ldots,f_k$, respectively and let $R_i$ be a (nonempty) bad set of $D_i$ for all $0\leq i\leq k$.
 Let $D$ be the (directed) graph built as follows:
 \begin{itemize}                                                                                                                                                                                                                                                                                                                                                                                                       \item replace vertex $v_i$ of $D_0$ by $D_i$, for $1\leq i\leq k$;                                                                                                                                                                                                                            \item for every arc (edge) $v_iv_j$ of $D_0$ add to $D$ all arcs (edges) going from vertices of $R_i\subseteq D_i$ towards vertices of $R_j\subseteq D_j$.                                                                                                                                                                                                                            \end{itemize}
Then $D$ has order $n_1+\ldots+n_k$, minimum feedback vertex set of size at least $f_0+f_1+\ldots+f_k$. If $R_i$ is minimal (inclusion-wise) for $1\leq i\leq k$, then $f(D)=\sum_{0\leq i\leq k}f_i$ and $R=\bigcup_{v_i\in R_0}R_i$ is a bad set of $D$.

Moreover, if $R_1,\ldots,R_k$ cliques of size at most $c$, $D_1,\ldots,D_k$ are chordal, have treewidth at most $t$, and $D_0$ has treewidth at most $t_0$, then $D$ has treewidth  at most $\max(t,(c+1)t_0-1)$.
\end{proposition}
 
\begin{figure}[H]
    \centering
    \includegraphics[width=.4\textwidth]{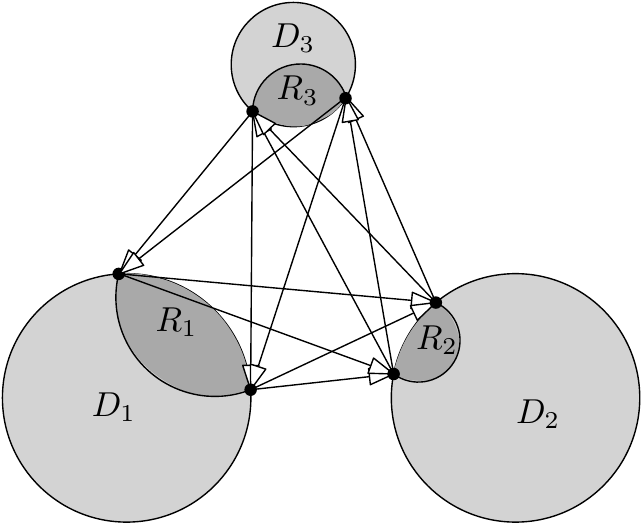}
    \caption{The construction of $D$ in Proposition~\ref{prop:gen2} where $D_0$ is a directed triangle.}\label{fig:metaD}
\end{figure} 

\begin{proof}
Clearly, $D$ has order $n_1+\ldots+n_k$. Let $F$ be an optimal feedback vertex set of $D$ and let $F_i=F\cap V(D_i)$ for $i\in [k]$. Suppose (after relabeling) that $|F_i|=f_i$ for $i\leq \ell\leq k$ and $|F_i|>f_i$ otherwise. Since we glued on bad sets $R_i$, the (di)graph $D-F$ has an induced acyclic sub(di)graph isomorphic to an $\ell$-vertex acyclic sub(di)graph of $D_0$. A maximum acyclic set in $D_0$ has size at most $k-f_0$ by definition of $f_0$. So, we have $\ell\leq k-f_0$. Finally, we get $f(D) = \sum_{i\in [k]}|F_i|\geq \sum_{i\in [k]}f_i +k-\ell \geq \sum_{i\in [k]}f_i +k-(k-f_0) = f_0+f_1+\dots+f_k$. 
 
 We claim that $f(D)=\sum_{0\leq i\leq k}f_i$ if $R_i$ is minimal (inclusion-wise) for $1\leq i\leq k$. Indeed, for each $i$, there exists an optimal feedback vertex set $F'_i$ of $D_i$ such that $|R_i\setminus F'_i|=1$ by minimality of $R_i$. So, $\bigcup_{1\leq i\leq k}(R_i\setminus F'_i)$ is isomorphic to $D_0$ and has a feedback vertex set $F'_0$ of size $f_0$. Therefore, $F'=\bigcup_{0\leq i\leq k}F'_0$ is a feedback vertex set of $D$ of size $\sum_{0\leq i\leq k}f_i$.    

Observe that, if $F$ is a feedback vertex set of $D$, then $F_0=\{v_i\in V(D_0):R_i\subset F\}$ is a feedback vertex set of $D_0$. Moreover, $|F|\geq |F_0|+f_1+\dots+f_k$. As a result, if $R_0$ is a bad set of $D_0$ and a feedback vertex set $F$ of $D$ contains $\bigcup_{\{i:v_i\in R_0\}}R_i$, then $F_0$ is not a minimum feedback vertex set and $|F_0|>f_0$. This yields $|F|> f_0+f_1+\dots+f_k=f(D)$ when $R_i$ is minimal for $1\leq i\leq k$. Hence, $\bigcup_{\{i:v_i\in R_0\}}R_i$ is a bad set of $D$. 
 
 The bound on the treewidth comes simply from eliminating in each graph $D_i$ all vertices different from $R_i$ first. Afterwards, we are left with the graph obtained from $D$ by replacing each vertex $v_i$ by a clique of size $R_i$. It is straight-forward to check that its treewidth is at most $ct_0+(c-1)$, where $t_0=\tw(D_0)$  and $c=\max\{|R_i| \mid i\in[k]\}$.
\end{proof}

For the next construction we consider a triple $(D,R,r')$ of a (directed) graph  $D=(V,A)$, with a bad set $R\subseteq V$ and a vertex $r'\in V\setminus R$. Denote by $D_{r'\times |R|}=(V',A)$ the (directed) graph obtained from $D$ by replacing $r '$ by a stable set $S$ of size $|R|$ each of whose vertices is connected to $D$ the same way as $r'$. 
The \emph{right-left-degeneracy} $\dg_{\mathrm{RL}}(D,R,r')$ of $(D,R,r')$ is the minimum $k$ such that $D_{r'\times |R|}$ has a $S$-last $k$-elimination ordering. 

If $D$ is chordal and $R$ is a clique, denote by $D'_{r'\times |R|}=(V',A)$ the (directed) graph obtained from $D$ by replacing $r '$ by a clique $S$ (oriented arbitrarily) of size $|R|$ each of whose vertices is connected to $D$ the same way as $r'$, then the 
$D'_{r'\times |R|}$ has an $N[S]$-last chordal $k$-elimination ordering.

\begin{proposition}\label{prop:gen}
 Let $(D,R,r')$ be a building block such, such that $D$ has $n$ vertices, minimum feedback vertex set of size $f$, and $R$ is bad. Then there is a family $(D_i)_{i\in \mathbb{N}}$ of (directed) graphs such that:
 \begin{itemize}
\item $n(D_i)=n+i(n-1)$,
\item $f(D_i)\geq f+if$ (with equality when $R$ is minimal inclusion-wise),                                                                                                                                                                                                                                                      \item $\dg(D_i)\leq\dg_{\mathrm{RL}}(D,R,r')$,
\end{itemize}
\end{proposition}

\begin{figure}[H]
    \centering
    \includegraphics[scale=0.9]{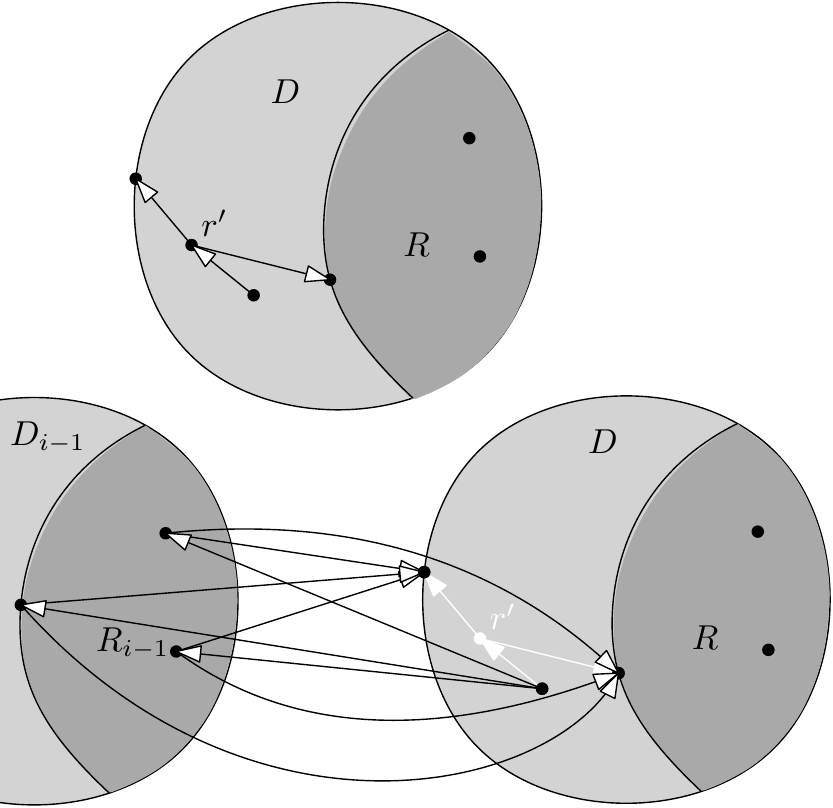}
    \caption{The construction of $D_{i}$ from $D_{i-1}$ and $D$ in  Proposition~\ref{prop:gen}.}\label{fig:generalconstruction}
\end{figure}

\begin{proof}
The first member of the family is $D_0=D$. We will construct $D_i$ by gluing a copy of $D$ to $D_{i-1}$ in a special way. We call $R_{i-1}$ the set of vertices corresponding to the vertices of $R$ in the last copy of $D$ in $D_{i-1}$.
Formally, the graph $D_i$ is built from $D_{i-1}$ as follows:
\begin{enumerate}
\item Take $D_{i-1}$ and a copy of $D$, and $r'\in V(D)$ as in the statement of the lemma,
\item add an arc from $(r,v)$ (resp. $(v,r)$) for all $r\in R_{i-1}\subseteq V(D_{i-1})$ and $v\in V(D)$ if $(r',v)\in A(D)$ (resp. $(v,r')\in A(D)$)
\item Delete $r'\in V(D)$ from the newly created graph.
\end{enumerate}
See Figure~\ref{fig:generalconstruction} for an illustration.
 
Using the induction hypothesis we have $n(D_i)=n(D_{i-1})+(n-1)=n+(i-1)(n-1)+(n-1)=n+i(n-1)$. 

Now, suppose that $F$ is a minimum feedback vertex set of $D_i$. Let $F'=F\cap D_{i-1}$ and $F''=F\setminus F'$. Observe that $F'$ must be a feedback vertex set of $D_{i-1}$. If $|F'|=f(D_{i-1})$, then, since $R_{i-1}$ is a bad set of $D_{i-1}$, there is a vertex $r\in R_{i-1}\setminus F'$. We denote by $(D-\{r'\})\cup\{r\}$ the  copy of $D$ where we removed $r'$ and added $r$ with all of the corresponding arcs as in the second item from the description above. So, $F''$ must be a feedback vertex set of $(D-\{r'\})\cup\{r\}$ which is isomorphic to $D$ and $|F''|\geq f$. Thus, $|F| = |F'|+|F''|\geq f(D_{i-1})+f=f+(i-1)f+f=f+if$. If $|F'|>f(D_{i-1})$, then $F''$ is a feedback vertex set of $D-\{r'\}$ and $|F''|\geq f-1$. Thus, $|F| = |F'|+|F''| \geq f(D_{i-1})+1+f-1=f+if$. Finally, note that if the bad set $R$ of $D$ is contained in $F''$, then $|F''|>f$ and $F$ cannot be optimal. Thus, $R_i=R$ is a bad set for $D_i$. 

When $R$ is minimal (inclusion-wise), then taking a minimum feedback vertex set $F'$ (of size $f(D_{i-1})$) of $D_{i-1}$ and a minimum feedback vertex set $F''$ (of size $f$) disjoint from $F'$ of $(D-\{r'\})\cup\{r\}$ where $r$ is the unique vertex in $ R\setminus F'$ results in $F=F'\cup F''$, a feedback vertex set of size $f(D_{i-1})+f=f+if$ of $D_i$. 

Let $k=\dg_{\mathrm{RL}}(D,R,r')$. Since $D_{r'\times |R|}$ has a $S$-last $k$-elimination ordering, where $D_{r'\times |R|}=(V',A)$ is the (directed) graph obtained from $D$ by replacing $r '$ by a stable set $S$ of size $|R|$ each of whose vertices is connected to $D$ the same way as $r'$. This means that in $D_i$ we can eliminate the vertices from right to left, i.e, starting with the vertices of last added copy $D-\{r'\}$. The fact that every neighbor of $r'$ in $D$ now instead has $|R|$ neighbors in $D_{i-1}$ is accounted by replacing $r '$ by a stable set $S$ of size $|R|$ in $D_{r'\times |R|}$. After eliminating $D-\{r'\}$, we iterate the argument with $D_{i-1}$.
 
The same argument, works for the claim on the treewidth. The fact that $R$ is a clique of $D$, is propagated through the construction, i.e., $R_{i-1}$ is a clique of $D_{i-1}$. Now, if we have the chordal $N[S]$-last $k$-elimination ordering of $D'_{r'\times |R|}$, this yields a $k$-chordal elimination ordering of $D_i$, where we start with $D-\{r'\}$. In particular, since the ordering is $N[S]$-last and $S$ is a clique that is contained in the neighborhood of all $v\in N[S]$ we can assume, that an optimal $N[S]$-last elimination ordering first removes $N[S]\setminus S$ and then $S$. This allows to create the chordal elimination ordering for $D_i$.
\end{proof}

This construction with $R$ an edge and $D$ being a directed triangle led to the examples in~\cite{KVW17}. 

\begin{observation}
By \Cref{prop:gen}, if we have a building block $(D,R,r')$ with right-left-degeneracy $k=\dg_{\mathrm{RL}}(D,R,r')$, then we obtain directly the lower bound $f\geq \frac{f(D)}{n(D)-1}$ for the class of graphs of degeneracy $k$. 
\end{observation}

\section{Undirected graphs}
In this section we only consider undirected graphs. We begin by giving the known and the easy results in a bit more detail.
The \emph{acyclic chromatic number} of a graph $G$ is the smallest $\ell$ such that $G$ has a proper $\ell$-coloring such that every cycle has at least $3$ colors. In~\cite{FGR02} it is shown that the acyclic chromatic number of a graph of treewidth $k$ is at most $k+1$, and this is used to show that:
\begin{proposition}[\!\!\cite{FGR02}]\label{prop:un_tw_ub}
Let $G$ be a graph of treewidth $k$. Then $f(G)\leq \frac{k-1}{k+1}{n(G)}$. Moreover, for every $k$ there are graphs of treewidth $k$ with ${f(G)}=\frac{k-1}{k+1}{n(G)}$.
\end{proposition}

Thus, the case of bounded treewidth is solved for undirected graphs. However, in~\cite{KM76} it is shown that the acyclic chromatic number is unbounded on the class of graphs of bounded degeneracy. Hence, the strategy of~\cite{FGR02} cannot work on this larger class. However, the same upper bound as in \Cref{prop:un_tw_ub} is tight for odd degeneracy:
\begin{proposition}\label{prop:un_dg_trivial}
Let $G$ be a graph of degeneracy $k$. Then ${f(G)}\leq \frac{k-1}{k+1}{n(G)}$ if $k$ is odd and ${f(G)}\leq \frac{k}{k+2}{n(G)}$ if $k$ is even. 
\end{proposition}
\begin{proof}
Along a $k$-elimination ordering $\phi$ of $G$ one can inductively construct a coloring of $V$ into $\lceil\frac{k+1}{2}\rceil$ induced forests. Indeed, remove the right-most vertex $v$ in the ordering, color by induction, re-introduce $v$. Since $d_p(v)\leq k$, one color $c$ is used at most once in the neighborhood of $v$, and $v$ can be coloured $c$ to extend the forest of color $c$. Now, the union of any $\lceil\frac{k+1}{2}\rceil-1$ of the constructed forests is a feedback vertex set. Distinguishing the parity of $k$, we get the claimed upper bounds.

\end{proof}

So, the remainder of this section is about the highest ratio of minimum feedback vertex set and order on graphs of even degeneracy $k$. From Propositions~\ref{prop:un_tw_ub} and~\ref{prop:un_dg_trivial} we know that this value lies between $\frac{k-1}{k+1}$ and $\frac{k}{k+2}$.

A reason to believe that none of both bounds is tight, is the case $k=2$. It is shown in~\cite{BDBS14} that any graph $G$ of degeneracy $2$ has ${f(G)}\leq\frac{2}{5}{n(G)}$ and there are infinitely many $2$-degenerate graphs attaining this bound. The following presents a first construction showing, that the lower bound of $\frac{k-1}{k+1}$ is not tight. Note that for $k=2$, the construction coincides with the $2$-degenerate graph given in~\cite{BDBS14}.

\begin{proposition}\label{prop:un_dg_1stlb}
 For every even $k\geq 2$ there is a graph $G$ with $\dg(G)\leq k$, $n(G)=\frac{(k+2)k}{2}+1$ and $f(G)=\frac{k^2}{2}$.
\end{proposition}

\begin{figure}[H]
    \centering
    \includegraphics[width=\textwidth]{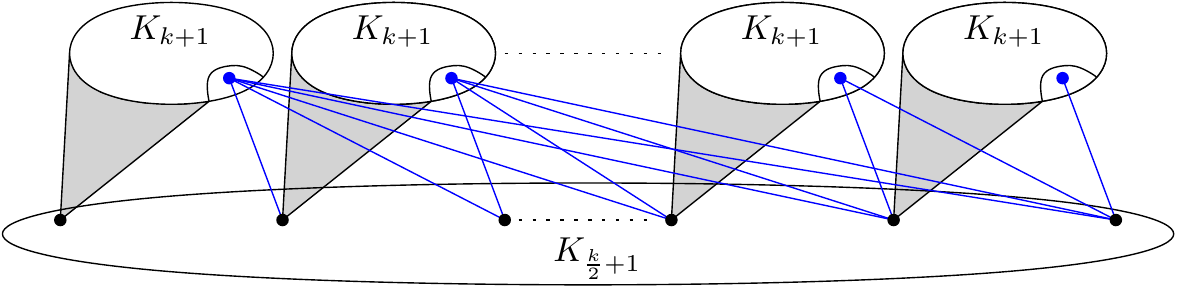}
    \caption{Constructing graphs of even degeneracy $k$, order $\frac{(k+2)k}{2}+1$, and minimum feedback vertex set $\frac{k^2}{2}$.}\label{fig:evendegundirectedlb}
\end{figure}

\begin{proof}
 See Figure~\ref{fig:evendegundirectedlb} for a construction of $G$. We take $\ell=\frac{k}{2}$ disjoint copies of complete graph $K_{k+1}$ on vertex set $[k+1]\times[\ell]$, and one $K_{\ell+1}$ on vertex set $[\ell+1]$. Add all edges from $[k]\times[i]$ to $i$ for all $i\in[\ell]$ -- these are the grey edges in Figure~\ref{fig:evendegundirectedlb}. Further, add all edges from $(k+1,i)$ to $j$ for all $1\leq i<j\leq \ell+1$ -- these are the blue edges in Figure~\ref{fig:evendegundirectedlb}.
 
First, note that $n(G)=\ell(k+1)+\ell+1=\frac{(k+2)k}{2}+1$. Second, let us show that one cannot build an optimal feedback vertex set of $G$ by choosing an optimal feedback vertex set in each of the cliques. Indeed, note that for any edge $ij\in K_{\ell+1}$ there is one of the $(k+1)$-cliques that is entirely contained in $N(\{i,j\})$. Thus, if we are optimum in $K_{\ell+1}$, then an edge $ij$ still remains after removing a feedback vertex. But then, if also an edge remains in the $(k+1)$-clique contained in $N(\{i,j\})$, the remaining graph contains a $C_3$ or a $C_4$. Hence, we need at least one vertex more than choosing an optimum feedback vertex set for each of the cliques. This yields that $f(G)\geq \ell(k-1)+(\ell-1)+1 = \frac{k^2}{2}$ as $\ell=\frac{k}{2}$.
 
 Third, let us prove $\dg(G)\leq k$. Note that vertex $\ell+1$ has $\ell$ neighbors in its clique and one neighbor in each of the $\ell$ cliques of order $k+1$. We thus can remove $\ell+1$. Now, the vertex $(k+1,\ell)$ has degree $k$ and can be removed. Now, all remaining vertices of $[k+1]\times\{\ell\}$ have degree $k$ and can be removed one after the other. Now, we remove vertex $\ell$ and continue similarly to remove all vertices. 
 
 \end{proof}
%

We now present a construction, that for large enough $n$ improves on the one from \Cref{prop:un_dg_1stlb}.
\begin{theorem}\label{thm:lbdg}
For every even $k$ there exists a family of $k$-degenerate graphs $(G_i)_{i\in \mathbb{N}}$ such that $n(G_i)=\frac{3k+6}{2}+i\frac{3k+4}{2}$ and $f(G_i)= \frac{3k-2}{2}+i\frac{3k-2}{2}$.
\end{theorem}

\begin{figure}[H]
    \centering
    \includegraphics[scale=1.3]{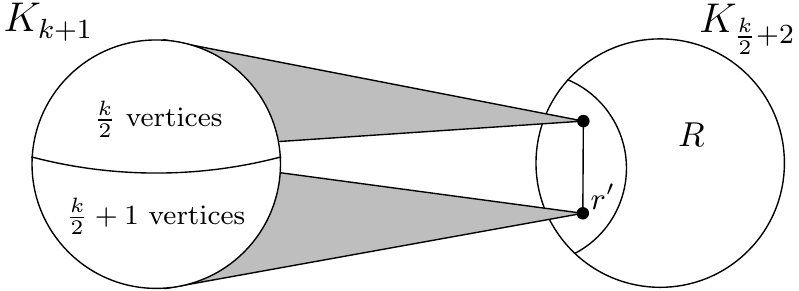}
    \caption{Graph $G_0$ of even degeneracy $k$, order $\frac{3k+6}{2}$, and minimum feedback vertex set $\frac{3k-2}{2}$.}\label{fig:evendegundirectedlb2}
\end{figure}

\begin{proof}
See \Cref{fig:evendegundirectedlb2} for a construction of $G_0$. We take two disjoint complete graphs: $K_{k+1}$ and $K_{\frac{k}{2}+2}$ on vertex set $[\frac{3k+6}{2}]$. Let $[k+1]$ be the vertex set of the first clique and $[k+2;\frac{3k+6}{2}]$ be the vertex set of the second. Add all edges from $[\frac{k+2}{2}]$ to $k+2$ and from $[\frac{k+4}{2};k+1]$ to $k+3$. Since $G_0$ can be partitioned into two disjoint cliques, a minimum feedback vertex set of $G_0$ must contain two disjoint optimal feedback vertex sets, one for each clique: the first $F'$ of size $k-1$ for $K_{k+1}$ and the second $F''$ of size $\frac{k}{2}$ for $K_{\frac{k}{2}+2}$. Thus, $f(G_0)\geq k-1+\frac{k}{2}=\frac{3k-2}{2}$. Observe that, if $F''$ contains $k+2$ (resp. $k+3$) and $F'$ contains all of the neighbors of $k+3$ (resp. $k+2$), then $F'\cup F''$ is a minimum feedback vertex set of $G_0$, so $f(G_0)=\frac{3k-2}{2}$. 

Let $R=[k+4;\frac{3k+6}{2}]$ and observe that $R$ is bad. Indeed, $R$ does not contain $k+2$ nor $k+3$, $R$ is contained in the clique of order $\frac{k}{2}+2$ and $|R|=\frac{k}{2}$. Similar to the construction in \Cref{fig:evendegundirectedlb}, $N(\{k+2,k+3\})$ contains the whole clique of order $k+1$ so an optimal feedback vertex set for $K_{k+1}$ of size $k-1$ would leave a triangle or a 4-cycle behind. Thus, a feedback vertex set of $G_0$ containing $R$ would have size at least $\frac{k}{2}+k-1+1=\frac{3k}{2}>f(G_0)$. 

As observed previously, a minimum feedback vertex set $F''$ of $K_{\frac{k}{2}+2}$ that contains $k+2$ or $k+3$ along with a corresponding minimum feedback vertex set $F'$ of $K_{k+1}$ form a minimum feedback vertex set of $G_0$. Therefore, for every $v\in R$, the set $F''=(R\setminus\{v\})\cup\{k+2\}$ has size $\frac{k}{2}$ and is such a minimum feedback vertex set. Thus, $R$ is minimal (inclusion-wise).

Let $r'=k+2$ and apply \Cref{prop:gen} on the triplet $(G_0,R,r')$. We obtain a family of graphs $(G_i)_{i\in \mathbb{N}}$ such that $n(G_i)=\frac{3k+6}{2}+i\frac{3k+4}{2}$ and $f(G_i)=\frac{3k-2}{2}+i\frac{3k-2}{2}$ since $R$ is minimal. Now, it suffices to prove that $\dg_{\mathrm{RL}}(G_0,R,r')\leq k$ to show that graphs of this family are $k$-degenerate. Consider the following elimination ordering of $(G_0)_{r'\times |R|}$ where we replace $r'$ with a stable set $S$ of size $|R|$ where every vertex of $S$ is connected to every vertex of $N(r')$:
\begin{itemize}
\item Start by removing vertices of $R$ - each has degree $\frac{k}{2}+1+|R|-1=\frac{k}{2}+1+\frac{k}{2}-1=k$.
\item Then, remove $k+3$ which now has degree $\frac{k}{2}+1+|R|-1=k$.
\item Afterwards, remove vertices of $N(k+3)$, each being of degree at most $k$ at this step.
\item Now, remove the vertices of $N(S)$, each of them being of degree $k$ at this step.
\item Finally, remove the independent set $S$.  
\end{itemize}
The above is an $S$-last $k$-elimination ordering of $(G_0)_{r'\times |R|}$, which concludes our proof.
\end{proof}

\Cref{thm:lbdg} shows that, for every $\epsilon>0$, there exists a $k$-degenerate graph on $n_{\epsilon}$ vertices for which $f\geq \frac{3k-2}{3k+4}n_{\epsilon} -\epsilon$.
We are not able to prove an asymptotically stronger upper bound than what is provided by Proposition~\ref{prop:un_dg_trivial}. However, we know that this bound cannot be attained with equality:
\begin{proposition}\label{prop:notequal}
Let $G$ be a graph of even degeneracy $k$. Then ${f(G)}<\frac{k}{k+2}{n(G)}$.
\end{proposition}
\begin{proof}
 Suppose to the contrary that $G$ with $f(G)=\frac{k}{k+2}{n(G)}$ exists. Note that the graph $G^+$ obtained from $G$ by adding an apex has odd degeneracy $k+1$ and therefore by Proposition~\ref{prop:un_dg_trivial}, we have $\frac{f(G^+)}{n(G^+)}=\frac{f(G^+)}{n(G)+1}\leq\frac{k}{k+2}$. This implies $f(G^+)=f(G)$. Thus, in $G$ there must be a minimum feedback vertex set $F$ such that $G^+- F$ has no cycle, in other words every cycle involving the apex in $G^+$ must intersect $F$. This means $V(G-F)$ is an independent set. But then $F$ was not minimum for $G$ since $G- (F\setminus\{v\})$ is a star forest for any $v\in F$, i.e., $F\setminus\{v\}$ is a smaller feedback vertex set of $G$.
\end{proof}

\section{Directed graphs}

In this section, we consider only directed graphs. The first part will be dedicated to the class of graphs with bounded degeneracy $k$. We start by showing that a minimum feedback vertex set $f$ of a directed graph of order $n$ is at most $\frac{k-1}{k+1}n$. Then we show some constructions with low degeneracy and high $\frac{f}{n}$ ratio. In the second part, we concentrate on the subclass of graphs with bounded treewidth $k$ where the additional structure allows an improvement on the upper bound of $f$ to $\frac{k}{k+3}n$. We will also show a general construction that improves previously known lower bounds for $f$ for any value of $k$.

\subsection{Degeneracy}

We start with a simple property on maximal $k$-degenerate graphs.

\begin{lemma} \label{adjacent_k-v}
There are no two adjacent $k$-vertices in a maximal $k$-degenerate graph with at least $k+2$ vertices.
\end{lemma}

\begin{proof}
Suppose by contradiction that $u$ and $v$ are two adjacent $k$-vertices in $G$, a maximal $k$-degenerate graph with at least $k+2$ vertices. Observe that, by definition, $d_p(v)=\min(k,\phi(v)-1)$ for every $v\in V(G)$ in any ordering of a maximal $k$-degenerate graph, so the first $k+1$ vertices always form a clique. Moreover, for a $k$-vertex $x\in V(G)$, we can always take an ordering and move $x$ to the last position to obtain another valid ordering.  

Thus, we can assume w.l.o.g. that there exists $\phi$ a $k$-elimination ordering of $G$ such that $\phi(u)<\phi(v)$ and $\phi(v)=|V(G)|$. We have $d(u)=d_p(u)+d_s(u)\geq d_p(u)+1$, so $d_p(u)\leq k-1$. Since $d_p(u)=\min(k,\phi(u)-1)$, we get $\phi(u)\leq k$. So, $u$ has $k$ neighbors in the first $k+1$ vertices which form a clique. These neighbors are all different from $v$ since $G$ has at least $k+2$ vertices and $v$ is the last vertex in $\phi$. This is a contradiction as $u$ is a $k$-vertex. 
\end{proof}

\begin{theorem}\label{thm:degdir}
Let $D$ be a $k$-degenerate directed graph, we have $f(D)\leq \frac{k-1}{k+1}n(D)$.
\end{theorem}

\begin{proof}
Observe that we already have this bound for $k$  odd from the undirected case (see \Cref{prop:un_dg_trivial}). So, we will prove \Cref{thm:degdir} for $k$ even. Suppose that $D$ is a counter-example minimizing the number of vertices and maximizing the number of edges. If $n(D)\leq k+1$, then any set of vertices of size $n(D)-2$ ($\leq \frac{k-1}{k+1}n(D)$) is a feedback vertex set of $D$. If $n(D)\geq k+2$, then $D$ contains a $k$-vertex since it is edge-maximal.

First, we show that $\delta^-(D)\geq \frac{k}{2}$ and $\delta^+(D)\geq \frac{k}{2}$. W.l.o.g. suppose by contradiction that there exists $u$ for which $d^-(u)\leq \frac{k}{2}-1$ since $k$ is even. Let $D'=D- (N^-(u)\cup\{u\})$, we have $f(D')\leq \frac{k-1}{k+1}n(D') = \frac{k-1}{k+1}(n(D)-(d^-(u)+1))$ by minimality of $D$. Take an optimal feedback vertex set $F$ of $D'$. Then $F\cup N^-(u)$ is a feedback vertex set of $D$ while $|F\cup N^-(u)| = f(D')+d^-(u)\leq \frac{k-1}{k+1}(n(D)-(d^-(u)+1))+d^-(u)=\frac{k-1}{k+1}n(D)+\frac{-(k-1)(d^-(u)+1)+(k+1)d^-(u)}{k+1} = \frac{k-1}{k+1}n(D)+\frac{2d^-(u)+1-k}{k+1}$. As $\frac{2d^-(u)+1-k}{k+1}< 0$ (since $d^-(u)\leq \frac{k}{2}-1$), we know that $|F\cup N^-(u)| < \frac{k-1}{k+1}n(D)$, which is a contradiction. As a result, for every $k$-vertex $u$, we have $k=d^-(u)+d^+(u)\geq\delta^-(D)+\delta^+(D)\geq k$. Therefore, $d^-(u)=d^+(u)=\frac{k}{2}$.

Now, we proceed as follows:

\textbf{Case 1:} \emph{$D$ contains two $k$-vertices sharing at least one neighbor}.
Let $u$ and $v$ be the two $k$-vertices and $w$ be adjacent to $u$ and $v$. If $w\in N^-(u)$ (resp. $N^+(u)$), then we remove $N^-(u)$ (resp. $N^+(u)$) from $D$. We do the same if $w\in N^-(v)$ (resp. $N^+(v)$), we also remove $u$ and $v$ from $D$ and call the resulting graph $D'$. Since $N^-(u)$, $N^+(u)$, $N^-(v)$, and $N^+(v)$ all have size $\frac{k}{2}$, we have $n(D')\geq n(D)-2\cdot\frac{k}{2}+1-2=n(D)-(k+1)$. Take an optimal feedback vertex set $F'$ of $D'$ to which we add every vertex removed from $D$ except for $u$ and $v$. We call the resulting set $F$. The set $F$ is a feedback vertex set of $D$ as each of $u$ and $v$ has become a source or a sink in $D- F$ and $|F| = f(D')+n(D)-n(D')-2 \leq \frac{k-1}{k+1}n(D')+n(D)-n(D')-2 =  n(D)-\frac{2}{k+1}n(D')-2$. Since $n(D')\geq n(D)-(k+1)$, we get $|F|\leq n(D)-\frac{2}{k+1}n(D')-2\leq \frac{k-1}{k+1}n(D)$, which is a contradiction. 

\textbf{Case 2:} \emph{$D$ contains a $k$-vertex adjacent to a $(k+1)^-$-vertex}. Let $u$ be the $k$-vertex and $v$ be its $(k+1)^-$-neighbor. Observe that since $D$ is edge-maximal, $v$ must be a $(k+1)$-vertex by \Cref{adjacent_k-v}. W.l.o.g. we assume that the arc between $u$ and $v$ is $(v,u)$. Since $\delta^-(D)\geq \frac{k}{2}$, $\delta^+(D)\geq \frac{k}{2}$, and $d(v)=k+1$, we must have either $d^-(v)=\frac{k}{2}$ or $d^+(v)=\frac{k}{2}$. If $d^-(v)=\frac{k}{2}$ (resp. $d^+(v)=\frac{k}{2}$), then let $D'=D-(N^-(v)\cup N^-(u)\cup\{u\})$ (resp. $D'=D-(N^+(v)\cup N^+(u)\cup\{v\})$). We have $n(D')\geq n(D) -2\cdot\frac{k}{2} - 1 = n(D)-(k+1)$. Take an optimal feedback vertex set $F'$ of $D'$ to which we add every vertex removed from $D$ except for $u$ and $v$. The resulting set $F$ is a feedback vertex set of $D$ since the arcs between $\{u,v\}$ and the rest of $D- F$ form a directed cut. Moreover, the same calculations as in \textbf{Case 1} yield $|F|\leq \frac{k-1}{k+1}n(D)$, which is a contradiction. 

Since by~\Cref{adjacent_k-v}, no two $k$-vertices in $D$ are adjacent, this covers all cases. Indeed, otherwise after removing all $k$-vertices from $D$, we obtain a graph with minimum degree at least $k+1$, contradicting that $D$ is $k$-degenerate. 
\end{proof}

The bound in \Cref{thm:degdir} is tight for $k=2$ e.g. a directed triangle. However, for greater values of $k$, we show that this bound is never reached when $k$ is odd. 

\begin{proposition}\label{prop:degdir}
Let $D$ be a directed graph of odd degeneracy $k\geq 3$, we have $f(D)<\frac{k-1}{k+1}n(D)$.
\end{proposition}

\begin{proof}
Take a counter-example $D$ to \Cref{prop:degdir} minimizing the number of vertices and maximizing the number of edges. In other words, we have an odd integer $k$ such that $D$ is $k$-degenerate and $f(D)\geq\frac{k-1}{k+1}n(D)$. Observe that $n(D)\geq k+1$, as otherwise for any set of vertices of size $n(D)-2$ ($<\frac{k-1}{k+1}n(D)$ as $n(D)<k+1$) is a feedback vertex set of $D$.

Since $D$ is an edge-maximal $k$-degenerate graph of order at least $k+1$, there exists a vertex $u$ of degree $k$. Consider the integer $l=\frac{k-1}{2}$. W.l.o.g. we can assume that $d^-(u)\leq d^+(u)$. Since $d^-(u)+d^+(u)=k=2l+1$, we have $d^-(u)\leq l$. Let $D'=D-(N^-(u)\cup\{u\})$, we have $f(D')<\frac{k-1}{k+1}n(D')=\frac{l}{l+1}n(D')$ by minimality of $D$. Take an optimal feedback vertex set $F$ of $D'$. Then, $F\cup N^-(u)$ is a feedback vertex set of $D$ while $|F\cup N^-(u)| = f(D')+d^-(u)<\frac{l}{l+1}n(D') + d^-(u)=\frac{l}{l+1}(n(D)-(d^-(u)+1))+d^-(u)=\frac{l}{l+1}n(D)+\frac{-l(d^-(u)+1)+(l+1)d^-(u)}{l+1} = \frac{l}{l+1}n(D)+\frac{d^-(u)-l}{l+1}$. As $\frac{d^-(u)-l}{l+1} \leq 0$ (since $d^-(u)\leq l$), we get $|F\cup N^-(u)|<f(D)$, which is a contradiction.  
\end{proof}

\Cref{prop:degdir} is optimal for $k=3$ since the construction from \cite{KVW17} shows that for every $n$, there exists a $3$-degenerate graph with a feedback vertex set of size at least $\lfloor\frac{n-1}{2}\rfloor$.

%
%

In the following table, we present some constructions, most of which use \Cref{prop:gen} to obtain good ratios for the minimum feedback vertex sets over number of vertices for graphs with low degeneracy.

\begin{table}[H]
\centering
\scalebox{0.95}{%
\bgroup
\def\arraystretch{1.5}%
\begin{tabular}{|c|l|c|c|}
\hline
degeneracy & building blocks for the lower bound using \Cref{prop:gen} 
& lower bound &  upper bound (\Cref{thm:degdir})\\\hline
3 & {directed triangle on $1,2,3$, $R=\{1,2\}$, $r'=3$ (as in~\cite{KVW17})} & $\frac{1}{3}$ $\longrightarrow \frac{1}{2}$ & $\frac{1}{2}$\\\hline
4 & \Cref{subfig:tw4}, $R=\{0,1\}$, $r'=2$ & $\frac{5}{10} \longrightarrow \frac{5}{9}$ &  $\frac{3}{5}$\\\hline
5 & \Cref{subfig:directeddeg4fvs4over8}, $R=\{0,1,6\}$, $r'=7$& $\frac{4}{8}$ $\longrightarrow \frac{4}{7}$ & $\frac{2}{3}$ \\\hline
6 & \Cref{subfig:directeddeg6twdth8fvs7over12}, $R=\{1,6,8\}$, $r'=10$ & $\frac{7}{12}$ $\longrightarrow \frac{7}{11}$ & $\frac{5}{7}$\\\hline
8 & \Cref{subfig:directeddeg7twdth7fvs6over10_bis}, $R=\{6,7,8\}$, $r'=0$& $\frac{6}{10}$ $\longrightarrow$ $\frac{6}{9}$ & $\frac{7}{9}$\\\hline
11 & \Cref{subfig:directeddeg9fvs7over11}, $R=\{0,1,4,8\}$, $r'=2$ & $\frac{7}{11}$ $\longrightarrow$ $\frac{7}{10}$ &  $\frac{5}{6}$ \\\hline
\end{tabular}
\egroup
}
\caption{\label{tab:directed_degen} Lower and upper bounds for largest ratio $\frac{f}{n}$ in digraphs with low degeneracy.}
\end{table}

\begin{figure}[!ht]
\centering
\subfloat[$\dg=4$, $f=5$, $n=10$, $R=\{0,1\}$, $r'=2$, $\dg_{\mathrm{RL}}=4$.\\ digraph6 encoding string: \texttt{IWWc?gbBAGET?W\_@`O}]{\label{subfig:tw4}

\quad\includegraphics[scale=0.43]{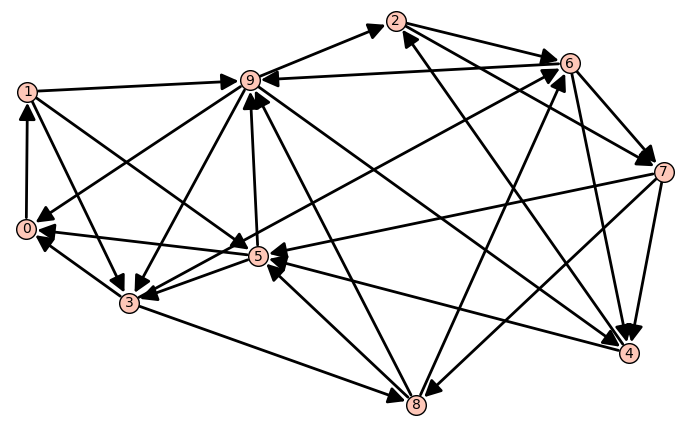}\quad
}
\qquad
\subfloat[$\dg=4$, $f=4$, $n=8$, $R=\{0,1,6\}$, $r'=7$, $\dg_{\mathrm{RL}}=5$.\\ digraph6 encoding string: \texttt{GDgJDW]@OI?o}]{\label{subfig:directeddeg4fvs4over8}

\quad\hspace{0.75cm}\includegraphics[scale=0.3]{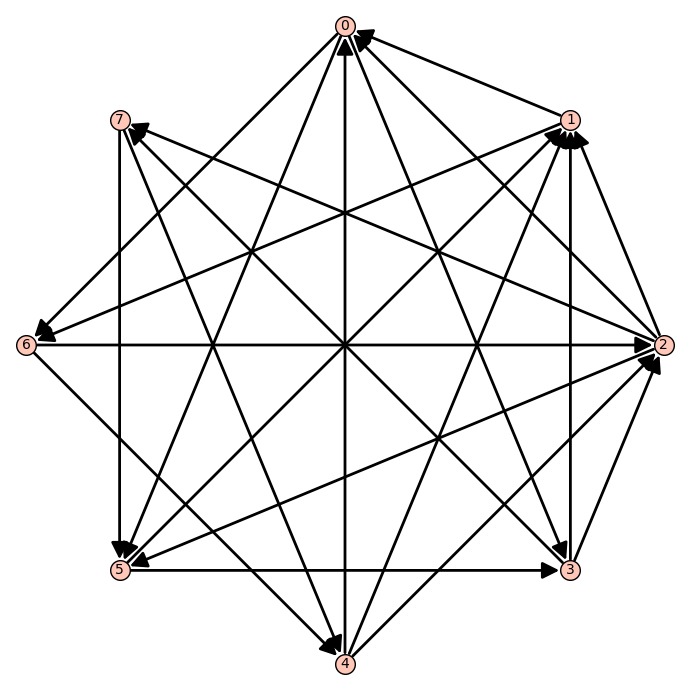}\hspace{0.75cm}\quad
}


\subfloat[$\dg=6$, $f=7$, $n=12$, $R=\{1,6,8\}$, $r'=10$, $\dg_{\mathrm{RL}}=6$.\\ digraph6 encoding string: \texttt{K]OL@DhAtH[ccOGGMtCw`B?\_Q}]{
\label{subfig:directeddeg6twdth8fvs7over12}
\quad\hspace{0.5cm}\includegraphics[scale=0.25]{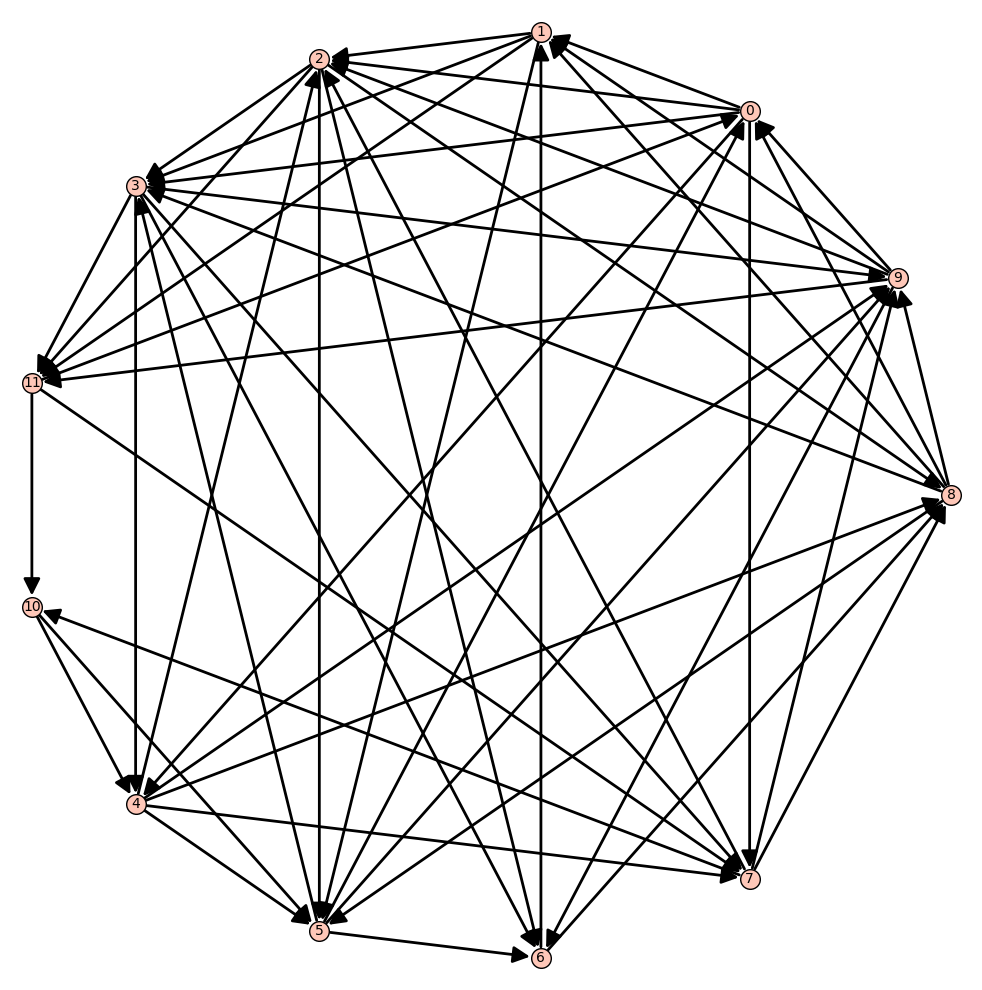}\hspace{0.5cm}\quad
}
\qquad
\subfloat[$\dg=7$, $f=6$, $n=10$, $R=\{6,7,8\}$, $r'=0$, $\dg_{\mathrm{RL}}=8$.\\ digraph6 encoding string: \texttt{IQ\_lhcpGUiM[OWy@\textbackslash\textbackslash ? }]{
\label{subfig:directeddeg7twdth7fvs6over10_bis}
\quad\hspace{0.5cm}\includegraphics[scale=0.35]{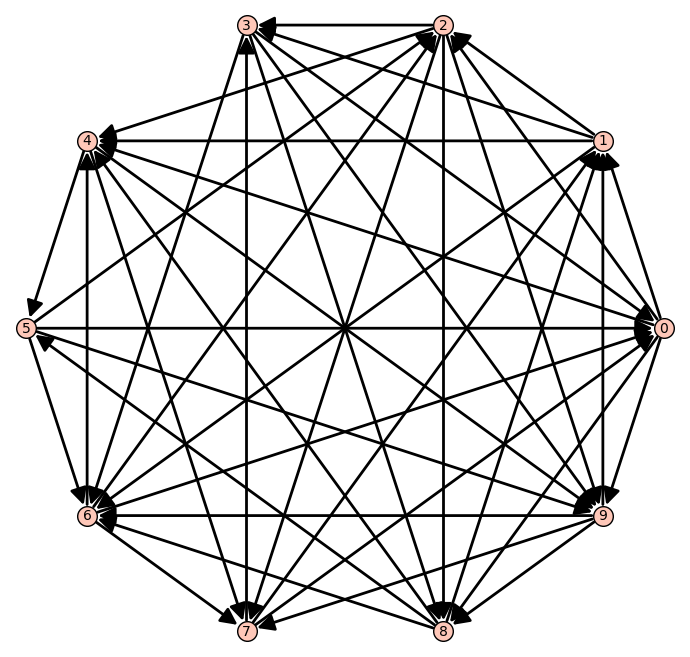}\hspace{0.5cm}\quad
}



\subfloat[$\dg=9$, $f=7$, $n=11$, $R=\{0,1,4,8\}$, $r'=2$, $\dg_{\mathrm{RL}}=11$.\\ digraph6 encoding string: \texttt{JTc\textbackslash\textbackslash c\textbackslash\textbackslash \_\textbackslash\textbackslash g\textbackslash\textbackslash g\textbackslash\textbackslash G\textbackslash\textbackslash G\^{}GRGZG?}]{
\label{subfig:directeddeg9fvs7over11}
\quad\hspace{0.5cm}\includegraphics[scale=0.25]{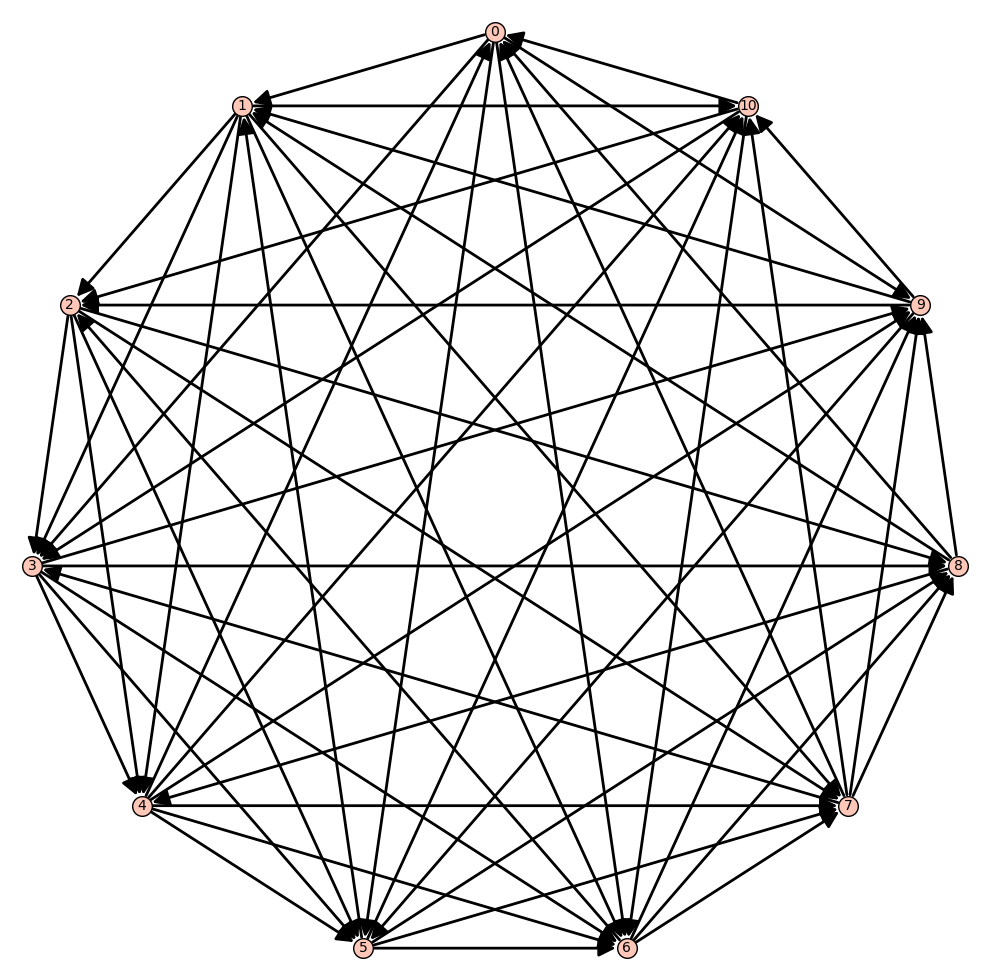}\hspace{0.5cm}\quad
}

\caption{Building blocks for \Cref{prop:gen} with low degeneracy and large minimum feedback vertex set.} 
\end{figure}

\subsection{Treewidth}

Before showing the upper bound in \Cref{thm:ubtw}, we give some simple lemmas and observations. First, 
by \Cref{tournaments}, we know that

\begin{observation} \label{undirected_triangle} Any set of four vertices of a directed graph contains an acyclic subset of size 3.
\end{observation}

\begin{lemma} \label{adjacent_k_k+1-v}
If $G$ is a $k$-tree with at least $k+4$ vertices, then there are no two adjacent $(k+1)$-vertices such that each of them is adjacent to a $k^-$-vertex.
\end{lemma}

\begin{proof}
Suppose there exist two such vertices $v$ and $w$ and call their respective $k^-$-neighbors $v'$ and $w'$. Observe that $G-\{v',w'\}$ is a $k$-tree with at least $k+2$ vertices and two adjacent $k^-$-vertices $v$ and $w$, which is a contradiction due to \Cref{adjacent_k-v}.
\end{proof}

\begin{lemma}\label{main_lemma}
In a $k$-degenerate graph $G$ there exists an integer $l\geq 1$ and a $(k+l)$-vertex $v$, such that $v$ has at least $l$ $k^-$-neighbors. Moreover, if $G$ is a $k$-tree, then $v$ has exactly $l$ $k$-neighbors. 
\end{lemma}

\begin{proof}
Let $G$ be a $k$-degenerate graph and consider $H=G-\{u|d(u)\leq k\}$. Suppose by contradiction that for all $l\geq 1$, every $(k+l)$-vertex $v\in V(G)$ has at most $l-1$ $k^-$-neighbors. As a result, every vertex in $H$ has degree at least $k+l-(l-1)=k+1$, which is a contradiction since $H$ is a $k$-degenerate graph. Now if $G$ is a $k$-tree, then so is $H$. Thus, if $v$ has more than $l$ $k$-neighbors in $G$, then $v$ would be of degree at most $k-1$ in $H$, which is a contradiction.
\end{proof}

\begin{lemma} \label{lem:clique_neighborhood}
Let $G$ be a $k$-tree. If $u$ is a $k$-vertex, then $N(u)\subseteq N[v]$ for every $v$ neighbor of $u$. 
\end{lemma}

\begin{proof}
Let $u$ be a $k$-vertex and $v\in N(u)$. Since $G$ is a $k$-tree, $N(u)$ is a $k$-clique. As a result, every neighbor of $u$ is a neighbor of $v$. In other words, $N(u)\subseteq N[v]$.
\end{proof}

\begin{theorem}\label{thm:ubtw}
If $G$ is a directed graph of treewidth $k$, then $f(G)\leq \frac{k}{k+3}n(G)$. 
\end{theorem}

\begin{proof}
First, observe that for $k=1$, $G$ is a tree so $f(G)=0$. As for $k=2$ (resp. $3$), we have proven that $f(G)\leq \frac{1}{3}n\leq \frac{2}{5}n$ (resp. $f(G)\leq \frac{2}{4}n = \frac{3}{6}n$) in \Cref{thm:degdir} as $G$ is also $k$-degenerate. Therefore, we only need to prove \Cref{thm:ubtw} for $k\geq 4$.

We proceed by induction on $n=n(G)$.

\paragraph{\textbf{Base case:}} We are going to prove that for every $0\leq n\leq k+6$, there exists a minimum feedback vertex set $F$ such that $|F|\leq \frac{k}{k+3}n$.

For $0\leq n\leq 3$, we always have $|F|\leq 1$ with equality if and only if $G$ is a directed triangle. As a result, we always have $|F|\leq \frac{k}{k+3}n$ for $k\geq 4$.

For $4\leq n\leq k+3$, by \Cref{undirected_triangle} we have $|F|\leq n-3$. Since $n-3-\frac{k}{k+3}n = 3\frac{n}{k+3}-3 = 3(\frac{n}{k+3}-1)\leq 0$, we get $|F|\leq \frac{k}{k+3}n$.

For bigger tournaments, by \Cref{cor:tournaments} we get the following:
\begin{itemize}
\item For $n=k+4$, $|F|\leq n - 4 = \frac{k}{k+4}n\leq \frac{k}{k+3}n$ since $k\geq 4$.

\item For $n=k+5$, $|F|\leq n - 4 = \frac{k+1}{k+5}n\leq \frac{k}{k+3}n$ since $k\geq 4$.

\item For $n=k+6$, if $4\leq k\leq 5$, then we checked by computer that directed $k$-trees on 10 (resp. 11) vertices have minimum feedback vertex sets of size at most $5$ (resp. 6) which gives $|F|\leq n-5 = \frac{k+1}{k+6}\leq \frac{k}{k+3}$.

If $k\geq 6$, then $|F|\leq n - 4 = \frac{k+2}{k+6}n\leq \frac{k}{k+3}n$. 
\end{itemize}

\paragraph{\textbf{Induction:}} Suppose that $n\geq k+7$.\\
For each case, we will define a subgraph $H$ of $G$ to which we will apply our induction hypothesis. Observe that in the following, we will never remove more than $k+7$ vertices at a time, thus, our base case along with the induction hypothesis will suffice to prove our theorem. 

Observe that when $G$ is a partial $k$-tree, we can add edges to $G$ to obtain a $k$-tree $G'$ and $f(G)\leq f(G')$ because a feedback vertex set of $G'$ is also a feedback vertex set of $G$. Thus, from now on, we can assume that $G$ is a $k$-tree. 

By \Cref{main_lemma}, there exists an integer $l\geq 1$ and a $(k+l)$-vertex $v$ such that $v$ has exactly $l$ $k$-neighbors $u_1,u_2,\dots,u_l$ since $G$ is a $k$-tree. In the following cases, $F$ will always denote a minimum feedback vertex set.

\paragraph{\textbf{Case 1:}} $l\geq 2$\\
Consider $H=G-(N_G[v]\setminus\{u_3,\dots,u_l\})$. Now, let $F(G)=F(H)\cup (N_G[v]\setminus\{v,u_1,\dots,u_l\})$. Observe that $G- F(G) = (H- F(H))\cup\{v,u_1,u_2\}$ so we have an acyclic set and a path on three vertices remaining due to \Cref{lem:clique_neighborhood} and \Cref{adjacent_k-v}. Moreover, $|F(G)|=|F(H)|+k+l+1-(l+1) = \frac{k}{k+3}(n-(k+l+1-(l-2)))+k = \frac{k}{k+3}(n-(k+3))+k = \frac{k}{k+3}n$.

\begin{center}
\begin{tikzpicture}[line cap=round,line join=round,>=triangle 45,x=1.0cm,y=1.0cm]
\draw (2,5)-- (3,2);
\draw (3,2)-- (4,5);
\draw (3,2)-- (10,5);
\draw [rotate around={0:(6,2)}] (6,2) ellipse (3.1cm and 0.79cm);
\draw (3.1,1) node[anchor=north west] {$N_G[v]\setminus\bigcup_{1\leq i\leq l}\{u_i\}\supseteq\bigcup_{1\leq i\leq l}N_G(u_i)$};
\draw (3.2,2.2) node[anchor=north west] {$ v $};
\draw (2.1,5.33) node[anchor=north west] {$u_1$};
\draw (4.1,5.33) node[anchor=north west] {$ u_2 $};
\draw (10.14,5.29) node[anchor=north west] {$u_l$};
\begin{scriptsize}
\fill [color=black] (3,2) circle (1.5pt);
\fill [color=black] (2,5) circle (1.5pt);
\fill [color=black] (4,5) circle (1.5pt);
\fill [color=black] (10,5) circle (1.5pt);
\end{scriptsize}
\end{tikzpicture}
\end{center}

\paragraph{\textbf{Case 2:}} $l = 1$\\
Let $I_k(G)=\{x|d_G(x)=k, xy \in E(G), d_G(y)=k+1\}$. Observe that $I_k(G)\neq\emptyset$ since $l = 1$. We define $G'=G- I_k(G)$. By applying \Cref{main_lemma} to $G'$, there exists an integer $l'\geq 1$ and a $(k+l')$-vertex $v'$ such that $v'$ has exactly $l'$ $k$-neighbors $u'_1,u'_2,\dots,u'_{l'}$ in $G'$ since $G'$ is a $k$-tree. Observe that there exists $1\leq j\leq l'$ such that $u'_j$ has a neighbor in $I_k(G)$. Otherwise $d_G(u'_j)=k$ and $u'_j\notin I_k(G)$ for every $1\leq j\leq l'$ and thus we get \textbf{Case 1} with $l=l'$ and $v=v'$. Let $u_j$ be the neighbor of $u'_j$ in $I_k(G)$ when it exists. Suppose w.l.o.g. that $u_1$ exists. 

\paragraph{\textbf{Case 2.1:}} $l'\geq 3$\\
Consider $H=G-(N_G[v']\setminus\{u'_4,\dots,u'_{l'},u_4,\dots,u_{l'}\})$. Now, let $F(G) = F(H)\cup(N_{G'}[v']\setminus\{u'_1,\dots,u'_{l'}\})$. Observe that $G- F(G) = (H- F(H))\cup\{u'_1,u'_2,u'_3,u_1,u_2,u_3\}$ so we have an acyclic set and a forest remaining due to \Cref{lem:clique_neighborhood}, \Cref{adjacent_k-v}, and \Cref{adjacent_k_k+1-v}. Moreover, $|F(G)|=|F(H)|+k+1 \leq \frac{k}{k+3}(n-(k+1+4))+k+1 = \frac{kn + 3 - k}{k+3} \leq \frac{k}{k+3}n $ for $k\geq 4$.

\begin{center}
\begin{tikzpicture}[line cap=round,line join=round,>=triangle 45,x=1.0cm,y=1.0cm]
\draw (2,5)-- (3,2);
\draw (3,2)-- (4,5);
\draw (3,2)-- (7,5);
\draw (3,2)-- (10,5);
\draw (2,5)-- (2,6);
\draw [dash pattern=on 3pt off 3pt] (4,5)-- (4,6);
\draw [dash pattern=on 3pt off 3pt] (7,5)-- (7,6);
\draw [dash pattern=on 3pt off 3pt] (10,5)-- (10,6);
\draw [dash pattern=on 3pt off 3pt] (3,2)-- (2,6);
\draw [dash pattern=on 3pt off 3pt] (3,2)-- (4,6);
\draw [dash pattern=on 3pt off 3pt] (3,2)-- (7,6);
\draw [dash pattern=on 3pt off 3pt] (3,2)-- (10,6);
\draw [rotate around={0:(6,2)}] (6,2) ellipse (3.1cm and 0.79cm);
\draw (2.2,1) node[anchor=north west] {$N_{G'}[v']\setminus\bigcup_{1\leq j\leq l'}\{u'_j\}\supseteq\bigcup_{1\leq j\leq l'}N_G(\{u'_j,u_j\})$};
\draw (3.2,2.2) node[anchor=north west] {$ v' $};
\draw (2.2,5.33) node[anchor=north west] {$u'_1$};
\draw (4.1,5.33) node[anchor=north west] {$ u'_2 $};
\draw (7.1,5.33) node[anchor=north west] {$ u'_3 $};
\draw (10.1,5.33) node[anchor=north west] {$u'_{l'}$};
\draw (2.1,6.33) node[anchor=north west] {$u_1$};
\draw (4.1,6.33) node[anchor=north west] {$u_2$};
\draw (7.1,6.33) node[anchor=north west] {$u_3$};
\draw (10.1,6.33) node[anchor=north west] {$u_{l'}$};
\begin{scriptsize}
\fill [color=black] (3,2) circle (1.5pt);
\fill [color=black] (2,5) circle (1.5pt);
\fill [color=black] (4,5) circle (1.5pt);
\fill [color=black] (7,5) circle (1.5pt);
\fill [color=black] (10,5) circle (1.5pt);
\fill [color=black] (2,6) circle (1.5pt);
\draw [color=black] (4,6) circle (1.5pt);
\draw [color=black] (7,6) circle (1.5pt);
\draw [color=black] (10,6) circle (1.5pt);
\end{scriptsize}
\end{tikzpicture}
\end{center}

\paragraph{\textbf{Case 2.2:}} $l'= 2$\\
If $u_2$ exists, then observe that by defining $H$ and $F(G)$ the same way as previously, the upper bound on $|F(G)|$ still holds and we have the same desired properties. So, suppose that $u_2$ does not exists. Since $|N_{G'}[v']\setminus\{u'_1,u'_2\}|=k+1$ and $d_{G'}(u'_1)=k$, there exists $w'\in (N_{G'}[v']\setminus\{u'_1,u'_2\})\setminus N_{G'}(u'_1)$. Now, consider $H=G-(N_G[v']\setminus\{w'\})$ and let 
$F(G) = F(H) \cup (N_G[v']\setminus\{w',u'_1,u'_2,u_1\})$. Observe that 
$ G- F(G) = (H- F(H))\cup\{u'_1,u'_2,u_1\}$ so we have an acyclic set to which we might attach a 1-vertex $u'_2$ and a path due to \Cref{adjacent_k-v}. Moreover, $ |F(G)| = |F(H)|+k = \frac{k}{k+3}(n-(k+3))+k = \frac{k}{k+3}n$.

\begin{center}
\begin{tikzpicture}[line cap=round,line join=round,>=triangle 45,x=1.0cm,y=1.0cm]
\draw (2,5)-- (3,2);
\draw (3,2)-- (4,5);
\draw (2,5)-- (2,6);
\draw [dash pattern=on 3pt off 3pt] (3,2)-- (2,6);
\draw [dash pattern=on 3pt off 3pt] (5,2)-- (4,5);
\draw [rotate around={0:(6,2)}] (6,2) ellipse (3.1cm and 0.79cm);
\draw (3,1) node[anchor=north west] {$N_{G'}[v']\setminus\{u'_1,u'_2\}\supseteq N_{G}(\{u'_1,u'_2,u_1\})$};
\draw (3.2,2.2) node[anchor=north west] {$ v' $};
\draw (5.1,2.2) node[anchor=north west] {$w'$ ($u'_1w',u_1w'\notin E(G)$)};
\draw (2.2,5.33) node[anchor=north west] {$u'_1$};
\draw (4.1,5.33) node[anchor=north west] {$ u'_2 $};
\draw (2.1,6.33) node[anchor=north west] {$u_1$};
\begin{scriptsize}
\fill [color=black] (3,2) circle (1.5pt);
\fill [color=black] (2,5) circle (1.5pt);
\fill [color=black] (4,5) circle (1.5pt);
\fill [color=black] (2,6) circle (1.5pt);
\fill [color=black] (5,2) circle (1.5pt);
\end{scriptsize}
\end{tikzpicture}
\end{center}

\paragraph{\textbf{Case 2.3:}} $l'= 1$\\
Let $I_k(G')=\{x'|d_{G'}(x')=k, x'y' \in E(G'), d_{G'}(y')=k+1\}$. Observe that $I_k(G')\neq\emptyset$ since $l'=1$. We define $G''=G'- I_k(G')$. By applying \Cref{main_lemma} to $G''$, there exists an integer $l''\geq 1$ and a $(k+l'')$-vertex $v''$ such that $v''$ has exactly $l''$ $k$-neighbors $u''_1,u''_2,\dots,u''_{l''}$ in $G''$ since $G''$ is a $k$-tree. Observe that there exists $1\leq j\leq l''$ such that $u''_j$ has a neighbor $u'_j$ in $I_k(G')$ and $u'_j$ has a neighbor $u_j$ in $I_k(G)$. Otherwise, we would get \textbf{Case 2.1} or \textbf{Case 2.2}. Suppose w.l.o.g. that $u'_1$ and $u_1$ exists. For each $1\leq j\leq l''$ such that $u'_j$ and $u_j$ exists, we define $v_j\in \{v'',u''_j,u'_j,u_j\}$ as the vertex, such that $\{v'',u''_j,u'_j,u_j\}\setminus\{v_j\}$ do not form a directed triangle, which exists thanks to \Cref{undirected_triangle}. 

\paragraph{\textbf{Case 2.3.1:}} $l''\geq 2$\\
Consider $H=G-(N_G[v'']\setminus\{u''_3,\dots,u''_{l''},u'_3,\dots,u'_{l''},u_3,\dots,u_{l''}\})$. Now, let $F(G) = F(H)\cup \{v_1,v_2\} \cup (N_{G''}(v'')\setminus\{u''_1,\dots,u''_{l''}\})$. Observe that $G- F(G) = (H- F(H))\cup(\bigcup_{j=1}^2\{v'',u''_j,u'_j,u_j\}\setminus\{v_j\})$ so we have two acyclic sets remaining due to \Cref{lem:clique_neighborhood}, \Cref{adjacent_k-v}, \Cref{adjacent_k_k+1-v} and the choices of the $v_j$'s. Let $i$ ($=1$ or $2$) be the number of $v_j$'s that exists for $1\leq j\leq 2$. We have $|F(G)|=|F(H)|+k+i \leq \frac{k}{k+3}(n-(k+1+2i+2))+k+i = \frac{kn + (3 - k)i}{k+3} \leq \frac{k}{k+3}n $ for $k\geq 4$.

\begin{center}
\begin{tikzpicture}[line cap=round,line join=round,>=triangle 45,x=1.0cm,y=1.0cm]
\draw (2,5)-- (3,2);
\draw (3,2)-- (4,5);
\draw (3,2)-- (10,5);
\draw (2,5)-- (1,6);
\draw (2,7)-- (1,6);
\draw [dash pattern=on 3pt off 3pt] (4,5)-- (3,6);
\draw [dash pattern=on 3pt off 3pt] (10,5)-- (9,6);
\draw [dash pattern=on 3pt off 3pt] (4,7)-- (3,6);
\draw [dash pattern=on 3pt off 3pt] (10,7)-- (9,6);
\draw [dash pattern=on 3pt off 3pt] (3,2)-- (1,6);
\draw [dash pattern=on 3pt off 3pt] (3,2)-- (3,6);
\draw [dash pattern=on 3pt off 3pt] (3,2)-- (9,6);
\draw [dash pattern=on 3pt off 3pt] (3,2)-- (2,7);
\draw [dash pattern=on 3pt off 3pt] (3,2)-- (4,7);
\draw [dash pattern=on 3pt off 3pt] (3,2)-- (10,7);
\draw [dash pattern=on 3pt off 3pt] (2,5)-- (2,7);
\draw [dash pattern=on 3pt off 3pt] (4,5)-- (4,7);
\draw [dash pattern=on 3pt off 3pt] (10,5)-- (10,7);
\draw [rotate around={0:(6,2)}] (6,2) ellipse (3.1cm and 0.79cm);
\draw (2,1) node[anchor=north west] {$N_{G''}[v'']\setminus\bigcup_{1\leq j\leq l''}\{u''_j\}\supseteq\bigcup_{1\leq j\leq l''}N_G(\{u''_j,u'_j,u_j\})$};
\draw (3.2,2.2) node[anchor=north west] {$ v'' $};
\draw (1.95,5.33) node[anchor=north west] {$u''_1$};
\draw (4.1,5.33) node[anchor=north west] {$ u''_2 $};
\draw (10.1,5.33) node[anchor=north west] {$u''_{l''}$};
\draw (0.4,6.33) node[anchor=north west] {$u'_1$};
\draw (2.4,6.33) node[anchor=north west] {$u'_2$};
\draw (9.1,6.33) node[anchor=north west] {$u'_{l''}$};
\draw (2,7.33) node[anchor=north west] {$u_1$};
\draw (4,7.33) node[anchor=north west] {$u_2$};
\draw (10.1,7.33) node[anchor=north west] {$u_{l''}$};
\begin{scriptsize}
\fill [color=black] (3,2) circle (1.5pt);
\fill [color=black] (2,5) circle (1.5pt);
\fill [color=black] (4,5) circle (1.5pt);
\fill [color=black] (10,5) circle (1.5pt);
\fill [color=black] (1,6) circle (1.5pt);
\draw [color=black] (3,6) circle (1.5pt);
\draw [color=black] (9,6) circle (1.5pt);
\fill [color=black] (2,7) circle (1.5pt);
\draw [color=black] (4,7) circle (1.5pt);
\draw [color=black] (10,7) circle (1.5pt);
\end{scriptsize}
\end{tikzpicture}
\end{center}

\paragraph{\textbf{Case 2.3.2:}} $l''= 1$\\
Since $|N_{G''}[v'']\setminus\{u''_1\}|=k+1$ and $d_{G''}(u''_1)=k$, there exists $w''\in (N_G[v'']\setminus\{u''_1\})\setminus N_G(u''_1)$. Recall that $v_1\in\{v'',u''_1,u'_1,u_1\}$ is such that $\{v'',u''_1,u'_1,u_1\}\setminus\{v_1\}$ does not form a directed triangle.

Now, consider $H=G-(N_G[v'']\setminus\{w''\})$. Let $F(G) = F(H) \cup \{v_1\}\cup (N_{G''}(v'')\setminus\{w'',u''_1\})$.
Observe that $G- F(G) = (H- F(H))\cup (\{v'',u''_1,u'_1,u_1\}\setminus\{v_1\})$ so we have an acyclic set by definition of $v_1$. Moreover, $|F(G)| = |F(H)|+k = \frac{k}{k+3}(n-(k+3))+k = \frac{k}{k+3}n$.

\begin{center}
\begin{tikzpicture}[line cap=round,line join=round,>=triangle 45,x=1.0cm,y=1.0cm]
\draw (2,5)-- (3,2);
\draw (2,5)-- (1,6);
\draw (2,7)-- (1,6);
\draw [dash pattern=on 3pt off 3pt] (3,2)-- (1,6);
\draw [dash pattern=on 3pt off 3pt] (3,2)-- (2,7);
\draw [dash pattern=on 3pt off 3pt] (2,5)-- (2,7);
\draw [rotate around={0:(6,2)}] (6,2) ellipse (3.1cm and 0.79cm);
\draw (3,1) node[anchor=north west] {$N_{G''}[v'']\setminus\{u''_1\}\supseteq N_G(\{u''_1,u'_1,u_1\})$};
\draw (3,2.2) node[anchor=north west] {$ v'' $};
\draw (1.95,5.33) node[anchor=north west] {$u''_1$};
\draw (3.9,2.2) node[anchor=north west] {$w''$ ($u''_1w',u'_1w'',u_1w'' \notin E(G)$)};
\draw (0.4,6.33) node[anchor=north west] {$u'_1$};
\draw (2,7.33) node[anchor=north west] {$u_1$};
\begin{scriptsize}
\fill [color=black] (3,2) circle (1.5pt);
\fill [color=black] (2,5) circle (1.5pt);
\fill [color=black] (1,6) circle (1.5pt);
\fill [color=black] (2,7) circle (1.5pt);
\fill [color=black] (3.8,2) circle (1.5pt);
\end{scriptsize}
\end{tikzpicture}
\end{center}
\end{proof}

This upper bound is tight for graphs of treewidth 3 as shown by the construction in~\cite{KVW17}.

Now, let us show some constructions improving the currently known lower bounds for directed graphs with bounded treewidth $k$ from \Cref{cor:tournaments} and \Cref{cor:EM}. 

\begin{lemma}\label{lemma:addedge}
Let $D=(V,A)$ be a directed graph and $R_D$ be a (inclusion-wise) minimal bad set  of $D$. Let $D'=(V',A')$, where $V'=V\cup\{r_1,r_2\}$ and $A'=A\cup\{(r_1,r_2)\}\cup\{(v,r_1),(r_2,v)|v\in R_D\}$, then $n(D')=n(D)+2$, $f(D') = f(D)+1$, and $\{r_1,r_2\}$ is a minimal bad set of $D'$. 
\end{lemma}

\begin{proof}
We build $D'$ by applying \Cref{prop:gen2}, where $D_0$ is a directed triangle, $D_1=D$, $R_1=R_D$, $D_2$ and $D_3$ are isolated vertices, $R_2=D_2$ and $R_3=D_3$. According to \Cref{prop:gen2} $R=\{r_1,r_2\}$ is a minimal bad set of $D'$ since any couple of vertices of $D_0$ is a minimal bad set of $D_0$. 
\end{proof}
 
\begin{theorem}\label{prop:lbtw}
For every $k$, there exists a family of directed graphs $(D_i)_{i\in\mathbb{N}}$ of treewidth $k$, such that $n(D_i)=k+2+i(k+1)$ and $f(D_i)\geq (i+1)(k-2\lfloor\log(k)\rfloor)$.
\end{theorem}

\begin{proof}
See \Cref{fig:lbtw}. Let $D$ be a tournament of order $k$ with $f(D)\geq k-2\lfloor\log(k)\rfloor-1$ from \Cref{cor:EM}. We build $D'$ by \Cref{lemma:addedge} and we know that $R=\{r_1,r_2\}$ is a bad set in $D'$ and $N(R)$ is a minimal bad set in $D$. By \Cref{lemma:addedge} we also know that $|N(R)|\leq k-1$, $n(D')=k+2$ and $f(D')\geq k-2\lfloor\log(k)\rfloor$. Take $r'$ to be any vertex of $D'- N[R]$ (which exists since $|V(D'- N[R])|\geq k+2-(k+1)=1$) and apply \Cref{prop:gen} to $(D',R,r')$ to obtain the family of graphs $(D_i)_i$ with the desired $n(D_i)$ and $f(D_i)$. As for its treewidth, first observe that $D'$ is chordal and $\{r_1,r_2\}$ is a clique. Now  the digraph $D'_{r'\times|R|}$, obtained from $D'$ by replacing $r'$ with a clique $S$ of order $|R|=2$, has an $N[S]$-last chordal $k$-elimination ordering. This is because $\{r_1,r_2\}$ has degree at most $k-1$ in $D'_{r'\times|R|}$ and $|N[S]|=k-1+2=k+1$.
\end{proof}

\begin{figure}[H]
\centering
\includegraphics[scale=1.4]{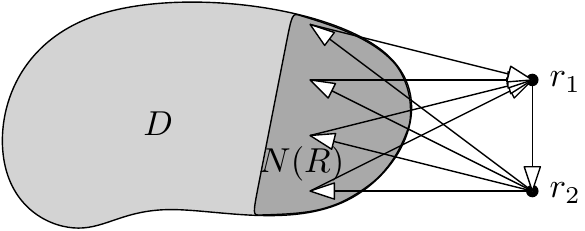}
\caption{\label{fig:lbtw} Building block for \Cref{prop:lbtw} where $D$ is a tournament on $k$ vertices with $f(D)\geq k-2\lfloor\log_2(k)\rfloor-1$.}
\end{figure}

\Cref{prop:lbtw} shows that, for every $\epsilon>0$, there exists a graph of treewidth $k$ on $n_{\epsilon}$ vertices for which $f\geq \frac{k-2\lfloor\log_2(k)\rfloor}{k+1}n_{\epsilon} -\epsilon$. This lower bound compared to $f\geq \frac{k-2\lfloor\log_2(k+1)\rfloor}{k+1}n$ from \Cref{cor:EM}, differs by $\frac{2}{k^2}n$ asymptotically for large values of $k$. Observe that this bound also holds for $2\leq k\leq 3$ as it gives $f\geq 0$ for $k=2$ and $f\geq \frac{1}{4}n$ for $k=3$.

\begin{proposition}\label{prop:lbtw2}
Given an integer $k$, a tournament $D$ of order $k+1$, two integers $l$ and $m\geq 1$ such that $2m+l=k+1$, and a tournament $D_0$ of size $l+m$, there exists a family of directed graphs $(D^i)_{i\in\mathbb{N^*}}$ of treewidth $k$ such that 
\begin{itemize}
\item $n(D^i)=(i+1)(k+1)$ and $f(D^i)\geq f(D)+i(f(D_0)+1)$ for $m=1$,
\item $n(D^i)=m^i(k+1+\frac{2m+l}{m-1})-\frac{2m+l}{m-1}$ and $f(D^i)\geq m^i(f(D)+\frac{m+f(D_0)}{m-1})-\frac{m+f(D_0)}{m-1}$ for $m\geq 2$. 
\end{itemize}

\end{proposition}

\begin{proof} 
Let $k,l,m,D,D_0$ be defined as in the statement. Let $T_1,\dots,T_m$ be copies of $D$ and $I_{m+1},\dots,I_{m+l}$ be isolated vertices. For each $T_i$, we build $D_i$ by \Cref{lemma:addedge} to obtain a minimal bad set $R_i$ of size 2. For $m+1\leq i\leq m+l$, let $D_i=R_i=I_i$. It is easy to see that $f(D_i)=0$, and that $R_i$ is bad and minimal. By \Cref{prop:gen2}, we get a graph of size $m(k+3)+l$ with a feedback vertex set of size $m(f(D)+1)+f(D_0)$. We call this graph $D^1$. See \Cref{fig:lbtw2} for an example with $k=4$, $D$ a tournament on 5 vertices and $f(D)=2$ (by \Cref{cor:tournaments}), $D_0$ a directed triangle, $l=1$, and $m=2$.

We can repeat this construction by adding two vertices to $D^1$ by \Cref{lemma:addedge} to obtain $D'^1$. By taking the same $D_0$, $m$ copies of $D'^1$, and $l$ isolated vertices. We can apply \Cref{prop:gen2} once again. By doing these operations iteratively, we get a family of directed graphs $(D^i)_{i\in\mathbb{N}}$ with order $n(D^i)=m\cdot (n(D^{i-1})+2)+l$ and a feedback vertex set of size $f(D^i)=m\cdot (f(D^{i-1})+1)+f(D_0)$. We get
\begin{itemize}
\item $n(D^i)=(i+1)(k+1)$ and $f(D^i)\geq f(D)+i(f(D_0)+1)$ for $m=1$,
\item $n(D^i)=m^i(k+1+\frac{2m+l}{m-1})-\frac{2m+l}{m-1}$ and $f(D^i)\geq m^i(f(D)+\frac{m+f(D_0)}{m-1})-\frac{m+f(D_0)}{m-1}$ for $m\geq 2$.
\end{itemize}

We claim that the constraint $2m+l=k+1$ ensures that there exists a chordal $k$-elimination ordering of any $D^i$. Indeed, it suffices to observe that there exists an $R$-last chordal $k$-elimination ordering of a tournament of size $k+1$ with with two added vertices by \Cref{lemma:addedge}. This means that we can start by removing every tournament glued to $D_0$. Then, since $m$ vertices of $D_0$, the tournament of size $m+l$, is replaced by $R_1,\dots,R_m$ which are cliques of size 2, and the other $l$ vertices remain single vertices, we also obtain a tournament of size $2m+l=k+1$. Thus, we can continue this chordal $k$-elimination ordering.

\end{proof}

\begin{figure}[H]
\centering
\includegraphics[scale=1.3]{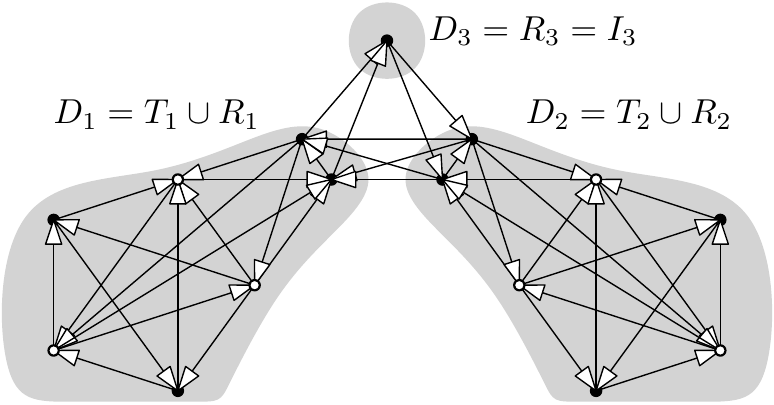}
\caption{\label{fig:lbtw2} Construction from \Cref{prop:lbtw2} for treewidth 4, where $D$ is a tournament on $5$ vertices and $D_0$ is a directed triangle (the white vertices form a bad set of the tournament).}
\end{figure}

\Cref{prop:lbtw2} shows that, for every $\epsilon>0$, there exists a directed graph of treewidth $k$ on $n_{\epsilon}$ vertices with minimum feedback vertex set $f$ such that:
$$f\geq \frac{f(D)+\frac{m+f(D_0)}{m-1}}{k+1+\frac{2m+l}{m-1}}n_{\epsilon} -\epsilon.$$

Recall that $D$ has size $k+1=2m+l$ and $D_0$ has size $m+l$. So, $\frac{\frac{m+f(D_0)}{m-1}}{\frac{2m+l}{m-1}}=\frac{m+f(D_0)}{k+1}\geq \frac{f(D)}{k+1}$ since the minimum feedback vertex set increases by at most 1 when $n$ increases by 1.
As a result, we always get a greater ratio than $\frac{f(D)}{k+1}$ since $\frac{\frac{m+f(D_0)}{m-1}}{\frac{2m+l}{m-1}}\geq \frac{f(D)}{k+1} \Leftrightarrow\frac{f(D)+\frac{m+f(D_0)}{m-1}}{k+1+\frac{2m+l}{m-1}}\geq \frac{f(D)}{k+1}$. The same also holds when we consider the case $m=1$ of \Cref{prop:lbtw2}. Moreover, taking tournaments from \Cref{cor:EM} of size $k+1$ and $m=\frac{k}{2}$, one can compute that \Cref{prop:lbtw2} yields a slightly better result than \Cref{prop:lbtw} with a difference of $\frac{4}{k^2}n$ instead of $\frac{2}{k^2}n$. The calculations are however tedious and we omit them here.

For smaller values of $k$, where we have the exact values of $f(D)$ and $f(D_0)$, the difference between the ratio given by tournaments and this construction for graphs with bounded treewidth is much more noticeable. In \Cref{tab:directed_tw}, for each treewidth $k$, we provide $n(D)$, $f(D)$, $n(D_0)$, $f(D_0)$, $l$, and $m$, as well as the lower bound on $\max\{\frac{f(H)}{n(H)}:H\text{ is a digraph of treewidth }k\}$ obtained by \Cref{prop:lbtw2}. We also include the lower bound obtained from tournaments by \Cref{cor:tournaments}. Finally, in the last column, we indicate the upper bound $\frac{k}{k+3}$ given by \Cref{thm:ubtw}.

\begin{table}[H]
\centering
\scalebox{0.9}{%
\bgroup
\def\arraystretch{1.5}%
\begin{tabular}{|c|c|c|c|c|c|c|c|c|c|}
\hline
$k$ & $n(D)$ & $f(D)$ & $n(D_0)$ & $f(D_0)$ & $l$ & $m$ & \makecell{lower bound \\by \Cref{cor:tournaments}} & \makecell{lower bound\\by \Cref{prop:lbtw2}} & \makecell{upper bound\\ by \Cref{thm:ubtw}}\\
\hline
4 & 5& 2& 3& 1& 1& 2& $\frac{2}{5}$ & $\frac{1}{2}$ & $\frac{4}{7}$ \\
\hline
5 & 6& 3& 3& 1& 0& 3& $\frac{1}{2}$ & $\frac{5}{9}$ & $\frac{5}{8}$ \\
\hline
6 & 7& 4& 3& 1& 0& 3& $\frac{4}{7}$ & $\frac{3}{5}$ & $\frac{6}{9}$ \\
\hline
7 & 7& 4& 7& 4& 6& 1& $\frac{4}{7}$ & $\frac{5}{8}$ & $\frac{7}{10}$ \\
\hline
8 & 7& 4& 7& 4& 5& 2& $\frac{4}{7}$ & $\frac{5}{8}$ & $\frac{8}{11}$ \\
\hline
9 & 10& 6& 7& 4& 4& 3& $\frac{6}{10}$ & $\frac{19}{30}$ & $\frac{9}{12}$ \\
\hline
10 & 11& 7& 7& 4& 3& 4& $\frac{7}{11}$ & $\frac{29}{44}$ & $\frac{10}{13}$ \\
\hline
11 & 12& 8& 7& 4& 2& 5&  $\frac{8}{12}$ & $\frac{41}{60}$ & $\frac{11}{14}$ \\
\hline
12 & 13& 9& 7& 4& 1& 6& $\frac{9}{13}$ & $\frac{55}{78}$ & $\frac{12}{15}$ \\
\hline
13 & 13& 9& 13& 9& 12& 1& $\frac{9}{13}$ & $\frac{5}{7}$ & $\frac{13}{16}$ \\
\hline
14 & 13& 9& 13& 9& 11& 2& $\frac{9}{13}$ & $\frac{5}{7}$ & $\frac{14}{17}$ \\
\hline
15 & 13& 9& 13& 9& 10& 3& $\frac{9}{13}$ & $\frac{5}{7}$ & $\frac{15}{18}$ \\
\hline
16 & 17& 12& 13& 9& 9& 4& $\frac{12}{17}$ & $\frac{49}{68}$ & $\frac{16}{19}$ \\
\hline
17 & 18& 13& 13& 9& 8& 5& $\frac{13}{18}$ & $\frac{11}{15}$ & $\frac{17}{20}$ \\
\hline
18 & 19& 14& 13& 9& 7& 6& $\frac{14}{19}$ & $\frac{85}{114}$ &  $\frac{18}{21}$ \\
\hline
19 & 20& 15& 13& 9& 6& 7& $\frac{15}{20}$ & $\frac{53}{70}$ & $\frac{19}{22}$ \\
\hline
20 & 21& 16& 13& 9& 5& 8& $\frac{16}{21}$ & $\frac{43}{56}$ & $\frac{20}{23}$ \\
\hline
21 & 22& 17& 13& 9& 4& 9& $\frac{17}{22}$ & $\frac{7}{9}$ & $\frac{21}{24}$ \\
\hline
22 & 23& 18& 13& 9& 3& 10& $\frac{18}{23}$ & $\frac{181}{230}$ & $\frac{22}{25}$ \\
\hline
23 & 24& 19& 13& 9& 2& 11& $\frac{19}{24}$ & $\frac{35}{44}$ & $\frac{23}{26}$ \\
\hline
24 & 25& 20& 13& 9& 1& 12& $\frac{20}{25}$ & $\frac{241}{300}$ & $\frac{24}{27}$ \\
\hline
25 & 26& 21& 13& 9& 0& 13& $\frac{21}{26}$ & $\frac{137}{169}$ & $\frac{25}{28}$ \\
\hline
26 & 27& 22& 13& 9& 0& 13& $\frac{22}{27}$ & $\frac{143}{175}$ & $\frac{26}{29}$ \\
\hline
27 & 27& 22& 27& 22& 26& 1&  $\frac{22}{27}$ & $\frac{23}{28}$ & $\frac{27}{30}$ \\
\hline
28 & 27& 22& 27& 22& 25& 2& $\frac{22}{27}$ & $\frac{23}{28}$ & $\frac{28}{31}$ \\
\hline
29 & 27& 22& 27& 22& 24& 3& $\frac{22}{27}$ & $\frac{23}{28}$ & $\frac{29}{32}$ \\
\hline
30 & 27& 22& 27& 22& 23& 4& $\frac{22}{27}$ & $\frac{23}{28}$ & $\frac{30}{33}$ \\
\hline
31 & 27& 22& 27& 22& 22& 5& $\frac{22}{27}$ & $\frac{23}{28}$ & $\frac{31}{34}$ \\
\hline
32 & 33& 27& 27& 22& 21& 6& $\frac{27}{33}$ & $\frac{163}{198}$ & $\frac{32}{35}$ \\
\hline
33 & 34& 28& 27& 22& 20& 7& $\frac{28}{34}$ & $\frac{197}{238}$ & $\frac{33}{36}$ \\
\hline
\end{tabular}
\egroup
}
\caption{\label{tab:directed_tw}Lower and upper bounds for largest ratio $\frac{f}{n}$ in digraphs with low treewidth.}
\end{table}

The \emph{directed Ramsey number $R(l)$} is the smallest integer such that all tournaments of order $R(l)$ contain a transitive subtournament (acyclic set) of order $l$, see~\cite{LP2021}. We have seen through computational experiments that for small values of $k$ ($k$ being the treewidth), taking $n(D_0)$ to be the largest $R(l)\leq k$ yields the best result with \Cref{prop:lbtw2}. However, it is not clear that it is the best candidate in general.

\section{Conclusion}
 We have presented upper and lower bounds for the minimum feedback vertex set in directed and undirected graphs of bounded degeneracy or treewidth. 
 In the undirected setting, while constructing non-trivial lower bounds, we were not able to essentially improve the easy upper bound for graphs of even degeneracy $k$. We suspect however that this is possible.
\begin{conjecture}
 There is an $\varepsilon>0$ such that the size $f$ of a minimum feedback vertex of every $n$-vertex graph of even degeneracy $k$ satisfies ${f}\leq (\frac{k}{k+2}-\varepsilon){n}$.
\end{conjecture}

In the directed setting, we obtained  upper bounds for digraphs of bounded treewidth and for digraphs of bounded degeneracy. We could however, not construct lower bounds that significantly improve over the probabilistic arguments for tournaments due to~\cite{EM64}. We think that our results can be improved.
\begin{conjecture}
 There is an $\varepsilon>0$ such that for every $n$ there exists an $n$-vertex digraph of degeneracy $k$ and minimum feedback vertex of size $f$ satisfying ${f}\geq \frac{k-(2-\varepsilon)\log_2(k+1)}{k+1}{n}$.
\end{conjecture}

\subsubsection*{Acknowledgements:} We thank Raphael Steiner and Felix Schr\"oder for some initial discussions on this subject during the workshop Order \& Geometry in Wittenberg 2020. Furthermore we are grateful for the fruitful discussions and computational assistance from Tristan Benoit. Part of this research was carried out during a stay of P.V. and H.L. at the institute of mathematics of the University of Barcelona (IMUB). 

This research was supported by the INS2I project \textit{ACDG - ACyclicité dans les (Di)Graphes}. K.K was partially supported by the Spanish \emph{Ministerio de Econom\'ia,
Industria y Competitividad} through grant RYC-2017-22701 and by MICINN through grant PID2019-104844GB-I00. P.V. was partially supported by Agence Nationale de la Recherche (France) under research grant ANR DIGRAPHS ANR-19-CE48-0013-01. Moreover K.K. and P.V. were partially supported by Agence Nationale de la Recherche (France) under the JCJC program (ANR-21-CE48-0012).

\bibliography{bib}
\bibliographystyle{plain}

\end{document}